\documentclass[12pt]{amsart}
\pdfoutput=1
\usepackage{graphicx}
\usepackage{amssymb}
\usepackage{amsmath}
\usepackage{amsthm}
\usepackage{amscd}
\usepackage{epsfig}

\allowdisplaybreaks[4]
\newtheorem{theorem}{Theorem}[section]

\newtheorem{claim}[theorem]{Claim}
\newtheorem{corollary}[theorem]{Corollary}
\newtheorem{lemma}[theorem]{Lemma}

\newtheorem{proposition}[theorem]{Proposition}

\newtheorem{remark}[theorem]{Remark}
\theoremstyle{definition}

\numberwithin{equation}{section}

\newcommand{\R}{\mathbb R}
\newcommand{\Z}{\mathbb Z}
\newcommand{\N}{\mathbb N}

\newcommand{\X}{\mathcal{X}}
\newcommand{\Y}{\mathcal{Y}}
\newcommand{\z}{\mathcal{Z}}

\newcommand{\e}{\varepsilon}

\def\p{\mathbb P}
\def \E {\mathbb E}
\begin{document}

\title{Variational principles for amenable metric mean dimensions}

\author [Ercai Chen, Dou Dou and Dongmei Zheng]{Ercai Chen, Dou Dou and Dongmei Zheng}

\address{School of Mathematical Sciences and Institute of Mathematics,
Nanjing Normal University, Nanjing 210023, Jiangsu, P.R.China} \email{ecchen@njnu.edu.cn}

\address{Department of Mathematics, Nanjing University,
Nanjing 210093, Jiangsu, P.R. China} \email{doumath@163.com}

\address{School of Physical and Mathematical Sciences, Nanjing Tech University,
Nanjing 211816, Jiangsu, P.R. China} \email{dongmzheng@163.com}


\begin{abstract}
In this paper, we prove variational principles between metric mean dimensions and rate distortion functions for countably infinite amenable group actions, which extends recent results by Lindenstrauss and Tsukamoto.
\if
In this paper, we prove two variational principles of metric mean dimensions for countably infinite amenable group actions. The first one is between metric mean dimensions and rate distortion functions, which extends recent results by Lindenstrauss and Tsukamoto from $\Z$-actions. The second one is between metric mean dimensions and measure-theoretic entropies.
\fi
\end{abstract}

\maketitle

\section{Introduction}
Entropy is the most successful invariant in dynamical systems which measures the complexity or uncertainty of the systems. It connects with information theory, dimension theory, fractal geometry and many other aspects in mathematics.

Due to the values of the entropy, dynamical systems can be divided into three classes:
1. systems with zero entropy; 2. systems with finite positive entropy; 3. systems with infinite entropy.
For zero entropy case, in order to give the quantitative measure of randomness or disorder, various of entropy type invariants were introduced:
sequence entropy (Kushnirenko \cite{Ku} and Goodman \cite{Go}), scaled entropy (Vershik \cite{V1,V2,V3}), entropy dimension (Carvalho \cite{Ca}, Ferenczi-Park \cite{FP} and Dou-Huang-Park \cite{DHP1,DHP2}) and so on. The studies on these invariants rely on the detailed analysis to the entropy-related
quantities or functions.
For infinite entropy case, the Gromov-Lindenstrauss-Weiss mean dimension is proved to be a meaningful quantity. The concept of mean dimension was first introduced
by Gromov \cite{Gr} in 1999 and then Lindenstrauss and Weiss \cite{LW} defined a metric version which is called metric mean dimension. These definitions of mean dimension
can be viewed as analogies of the concepts of dimension in dynamical systems. Mean dimension can be applied to solve imbedding problems in dynamical systems
(see for example, \cite{Gu1,Gu2,GT,L,LT}) and also supplies interesting quantities when characterizing large dynamics (\cite{T1,T2,T3}). In fact, from the definition, one may see easily that metric mean dimension is also an entropy-related quantity.

In the study of dynamical system and ergodic theory, people are always interested with the relationships between the topological concepts and the measure-theoretic ones.
For entropy, there exists a variational principle which says that topological entropy is the supreme of measure-theoretic entropy over all invariant Borel probability measures.
A natural question follows is does there exist variational principles for other entropy-related invariants?

For zero entropy case, it is shown that the traditional variational principle does not hold
for both sequence entropy (\cite{Go}) and entropy dimension (\cite{ADP}).

For infinite entropy case, people have been seeking variational principle of mean dimension for almost twenty years since
Gromov-Lindenstrauss-Weiss's mean dimension theory was established. In \cite{KD}, Kawabata and Dembo applied the rate-distortion function in information
 theory to investigate the dimension of fractal sets and established connections between dimensions and rate-distortion functions.
 Motivated by their work, recently Lindemstrauss and Tsukamoto \cite{LT2} proved variational principles for metric mean dimensions.
In the following we give a brief review of their results.

Let $(\X,d,T)$ be a TDS, where $\X$ is a compact metric space with metric $d$ and $T$ a continuous onto map from $\X$ to itself. Denote by
$M(\X,T)$ the collection of $T$-invariant Borel probability measures on $\X$. Let ${\rm \overline{mdim}_M}(\X,d)$ and ${\rm \underline{mdim}_M}(\X,d)$
be the upper and the lower metric mean dimension of TDS $(\X,d,T)$ respectively. Let $R_{\mu}(\cdot)$, $R_{\mu,p}(\cdot)$ and $R_{\mu,\infty}(\cdot)$
be the $L^1$, $L^p$ ($p>1$) and $L^{\infty}$ rate-distortion functions of $(\X,d,T)$ with respect to $\mu\in M(\X,T)$ respectively. For the definitions one may refer to \cite{LT2} and
we will also give the detailed definitions for amenable group actions in section 3.

Recall that the compact metric space $(\X,d)$ is said to have {\it tame growth of covering numbers} if for every $\delta>0$ it holds that
\begin{align}\label{tame-covering}
  \lim_{\e\rightarrow 0}{\e}^{\delta}\log \#(\X,d,\e)=0.
\end{align}

Lindemstrauss and Tsukamoto's variational principles are the following:

\begin{theorem}[$L^1$ and $L^p$ ($p>1$) variational principles, Theorem 16 and Corollary 11 of \cite{LT2}]
Let $(\X,d,T)$ be a TDS and $(\X,d)$ has tame growth of covering numbers, then
  $${\rm \overline{mdim}_M}(\X,d)=\limsup_{\e\to 0}\frac{\sup_{\mu\in M(\X,T)}R_{\mu}(\e)}{|\log \e|}=\limsup_{\e\to 0}\frac{\sup_{\mu\in M(\X,T)}R_{\mu,p}(\e)}{|\log \e|},$$
$${\rm \underline{mdim}_M}(\X,d)=\liminf_{\e\to 0}\frac{\sup_{\mu\in M(\X,T)}R_{\mu}(\e)}{|\log \e|}=\liminf_{\e\to 0}\frac{\sup_{\mu\in M(\X,T)}R_{\mu,p}(\e)}{|\log \e|}.$$
\end{theorem}

\begin{theorem}[$L^{\infty}$ variational principles, Theorem 9 of \cite{LT2}]
Let $(\X,d,T)$ be a TDS, then
  $${\rm \overline{mdim}_M}(\X,d)=\limsup_{\e\to 0}\frac{\sup_{\mu\in M(\X,T)}R_{\mu,\infty}(\e)}{|\log \e|},$$
$${\rm \underline{mdim}_M}(\X,d)=\liminf_{\e\to 0}\frac{\sup_{\mu\in M(\X,T)}R_{\mu,\infty}(\e)}{|\log \e|}.$$

\end{theorem}

Since many classic results including the SMB theorem and the variational principle for entropy have been generalized to actions by more larger class of groups beyond $\Z$ or $\Z^d$, it is natural
to ask whether the above variational principles still hold for such groups. In this paper we will work in the frame of countably infinite amenable group actions and establish the
 corresponding variational principles for amenable metric mean dimension. For the proofs we will follow Lindenstrauss and Tsukamoto's steps. Their proofs reveal that the method by
 Misiurewicz \cite{Mis} for proving the classical variational principle for entropy surprisingly works for metric mean dimension.
 But there are still additional difficulties for amenable group actions: when we construct the related invariant measures, we need some further tiling or quasi-tiling result for amenable groups (Lemma \ref{tiling-quasi2}) to
 produce some specific F{\o}lner sequence (Lemma \ref{lemma-folner}). To avoid complicated technical details, we employ the recent finite tiling result on amenable groups (Downarowicz et. \cite{DHZ}).

We would like to mention here that after Gromov-Lindenstrauss-Weiss's foundation works on mean dimension theory, there are sequences of articles on the theme for amenable mean dimensions.
See, for example \cite{C,CK,D,K1,K2,LiL}. There are also works for sofic group actions beyond amenable group actions \cite{H,Li,LiL2}. It will also be affirmative that whether there
exist variational principles for sofic mean dimensions.

The paper is organized as follows. In section 2, we will briefly recall the preliminaries for countably infinite amenable group including its tiling or quasi-tiling theory. And then prove our Lemma \ref{tiling-quasi2} and \ref{lemma-folner}. In section 3, we will introduce concepts and some properties for amenable metric mean dimensions, mutual information and amenable ($L^1$) rate-distortion functions.
Then in section 4 we will prove our ($L^1$) variational principles for amenable metric mean dimensions (Theorem \ref{theorem-L1}). In section 5, we will consider $L^\infty$ and $L^p$ ($p>1$) rate distortion functions and formulate the corresponding $L^\infty$ and $L^p$ ($p>1$) variational principles. Since the proof is parallel to the $L^1$ variational principles,
we leave it to Appendix A.

\section{Amenable groups and preliminary tiling lemmas}

Recall that a group $G$ is said to be {\it amenable} if there always exists an invariant Borel probability measure when it acts to any compact metric space.
In the case $G$ is a countable discrete group, amenability is equivalent to the existence of
a {\it F{\o}lner sequence}: a sequence of
finite subsets $\{F_n\}$ of $G$  such that
$$\lim_{n\rightarrow+\infty}\frac{|F_n\vartriangle gF_n|}{|F_n|}=0, \text{ for all } g\in G.$$
From now on, we always assume the group $G$ to be a countably infinite amenable group.

Denote by $F(G)$ the collection of nonempty finite subsets of $G$. Let $A,K\in F(G)$ and $\delta>0$. The set $A$ is said to be {\it $(K, \delta)$-invariant} if
\begin{align*}
  \frac{|B(A,K)|}{|A|}<\delta,
\end{align*}
where $B(A,K)$, the {\it $K$-boundary} of $A$, is defined by
$$B(A,K)=\{g\in G: Kg\cap A\neq\emptyset \text { and } Kg\cap(G\setminus A)\neq\emptyset\}.$$

Another equivalent condition for the sequence of finite subsets $\{F_n\}$ of $G$ to be a F{\o}lner sequence is that $\{F_n\}$ becomes more and more
invariant, i.e. for any $\delta>0$ and any finite subset $K$ of $G$, $F_n$ is $(K,\delta)$-invariant
for sufficiently large $n$. One may refer to Ornstein and Weiss \cite{OW} for more details on amenable groups, or Kerr and Li \cite{KL} for reference.


When considering amenable group actions in ergodic theory and dynamical systems, some kinds of ``tiling properties" are strongly involved in most of situations.
Not as good as the groups $\Z$ or $\Z^d$, in general it is still not known whether there always exist tiling F{\o}lner sets for all general amenable groups.
Ornstein and Weiss developed their quasi-tiling theory allowing some errors for the needed tiling properties and then many results for $\Z$ or $\Z^d$ actions can be extended to general amenable groups actions.

Let $\varepsilon\in(0,1)$. $A_1,A_2,\cdots,A_k\in F(G)$ are said to be {\it $\varepsilon$-disjoint} if there exist mutually disjoint $A_i'\subset A_i$ such that  $|A_i'|\ge (1-\e)|A_i|$ for $1\le i\le k$. We say $A_1,A_2,\cdots,A_k$ {\it $\varepsilon$-quasi-tile} $A\in F(G)$ if there exist $C_1,C_2,\cdots,C_k\subset F(G)$ (which are called the tiling centers) such that
\begin{enumerate}
\item for each $1\le i\le k$, $A_iC_i\subset A$ and $A_ic$'s for $c\in C_i$ are $\varepsilon$-disjoint,
\item for $1\le i\le k$, $A_iC_i$'s are mutually disjoint,
\item $|\cup_{i=1}^{k}A_iC_i|>(1-\varepsilon)|A|$.
\end{enumerate}
For $\varepsilon$-quasi-tiling we have the following simple observation:
 \begin{align}\label{ineq-quasi-tiling}
    (1-\varepsilon)|A|<\sum_{i=1}^k|A_i|\cdot|C_i|<\frac{1}{1-\e}|A|.
  \end{align}

The following is a fundamental quasi tiling theorem of amenable groups (see \cite{OW,WZ}).

\begin{theorem}\label{quasi-tiling}
Let $G$ be an amenable group and $\{e\}\subset F_1\subset F_2\subset\cdots$ be a F{\o}lner sequence in $G$. Then for
any $0<\epsilon<\frac{1}{4}$ and any integer $N>0$, there exist integers $N\le n_1<n_2<\cdots<n_k$ such that any $F_M$ (M sufficiently large)
can be $\epsilon$-quasi-tiled by $F_{n_1},F_{n_2},\cdots, F_{n_k}$.
\end{theorem}

\begin{remark}
  In the above theorem, the restriction on the F{\o}lner sequence $\{F_n\}$ can be removed (see \cite{ZCY}, Proposition 1) and the set $F_M$ can be replaced by
  any sufficiently invariant finite set (see \cite[Theorem 4.36]{KL} for reference).
\end{remark}


Recently, Downarowicz etc \cite{DHZ} proved a finite tiling result for general amenable groups. With the help of their result,
some of the proofs obtained from the quasi-tiling techniques can be simplified.

In the next let us recall the finite tiling result of Downarowicz etc \cite{DHZ}.

We call $\mathcal{T}\subset F(G)$ a {\it tiling} if $\mathcal{T}$ forms a partition of $G$. An element in a tiling $\mathcal{T}$ is called
a $\mathcal{T}$-tile or tile.
A tiling $\mathcal{T}$ is said to be finite if there exists a finite collection $\mathcal{S}=\mathcal{S}(\mathcal{T})=\{S_1, S_2,\ldots,S_k\}$ of $F(G)$,
which is called the shapes of $\mathcal{T}$, such that each element in $\mathcal{T}$ is a translation of some set in $\mathcal{S}$.
For convenience, we always assume that the shapes $\mathcal{S}$ has minimal cardinality, i.e. any set in $\mathcal{S}$ cannot
be a translation of others. Moreover, through some suitable translation, we can assume each set in $\mathcal{S}$ contain $e_G$.

Let $S$ be a shape of a finite tiling $\mathcal{T}$, the center of shape $S$ is the set $C(S)=\{c\in G: Sc\in \mathcal{T}\}$. For convenience,
we need $C(S)$ to be nonempty for each shape $S$.
We also require the centers $C(S)$'s satisfy that $Sc$'s are disjoint for $c\in C(S)$ and $S\in\mathcal{S}$.


For a tiling $\mathcal{T}$ with shapes $\mathcal{S}$, we can define a subshift $X_{\mathcal{T}}$ of $(\mathcal{S}\cup\{0\})^G$ by
$$X_{\mathcal{T}}=\overline{\bigcup_{g\in G}\{gx\}},$$
where $x=(x_g)_{g\in G}$ is defined by
\begin{align*}
  x_g=\begin{cases}
    S, \text{ if }g\in C(S),\\
    0, \text{ otherwise},
  \end{cases}
\end{align*}
i.e., $x$ is a transitive point of the subshift $X_{\mathcal{T}}$. We recall here that the shift action is defined by $(hx)_g=x_{gh}$ for $g,h\in G$.

Let $\mathcal{T}$ be a finite tiling of a countably infinite amenable group $G$. Denote by $h(\mathcal{T})=h_{top}(X_{\mathcal{T}},G)$, the topological entropy of the associated subshift $(X_{\mathcal{T}},G)$. The following is Theorem 5.2 of \cite{DHZ} by Downarowicz etc. Recall that a sequence of tiles $(\mathcal T_k)_{k\ge 1}$ is said to be {\it congruent} if for each
$k\ge 1$, every tile of $\mathcal T_{k+1}$ equals a union of tiles of $\mathcal T_k$.
\begin{theorem}
Let $G$ be a countably infinite amenable group. Fix a converging to zero sequence $\e_k>0$
and a sequence $K_k$ of finite subsets of $G$. There exists a congruent sequence of finite tilings ${\mathcal{T}}_k$ of $G$ such that the shapes
of $\mathcal{T}_k$ are $(K_k,\e_k)$-invariant and $h(\mathcal{T}_k)=0$ for each $k$.
\end{theorem}

In the present paper, we just need to use the following extract which is taken from Theorem 4.3 of Downarowicz etc \cite{DHZ}, a weaker version of the above theorem.

\begin{theorem}\label{tiling2}
For any $\e>0$ and $K\in F(G)$. There exists a finite tiling $\mathcal T$ of $G$, such that every shape of $\mathcal T$ is $(K,\e)$-invariant.
\end{theorem}

Recall that a F{\o}lner sequence $\{F_n\}$ in $G$ is said to be {\it tempered} if there exists a constant $C$ which is independent of $n$ such that
\begin{align}\label{tempered}
|\bigcup_{k<n}F_k^{-1}F_n|\le C|F_n|, \text{ for any }n.
\end{align}
Note that every F{\o}lner sequence $F_n$ has a tempered subsequence and in particular, every amenable group has a tempered F{\o}lner sequence (see Proposition 1.4 of Lindenstrauss \cite{L2}).

The following is the pointwise ergodic theorem for amenable group actions (Lindenstrauss \cite[Theorem 1.2]{L2}, see also Weiss \cite{W}).

\begin{theorem}[Pointwise Ergodic Theorem]\label{th-3-1}Let $(X,G,\mu)$ be an ergodic $G-$system, $\{F_n\}$ be a tempered F{\o}lner sequence in $G$ and $f\in L^1(X,\mathcal{B},\mu)$. Then
$$\lim_{n\rightarrow+\infty}\frac{1}{|F_n|}\sum_{g\in F_n}f(gx)=\int_{X}f(x)d\mu,$$
almost everywhere and in $L^1$.
\end{theorem}

Let $\mathcal{T}$ be a tiling of $G$ with shapes $\mathcal{S}=\{T_1,\ldots,T_l\}$. For $F\in F(G)$, $1\le j\le l$,
denote by $$\rho_{\mathcal T}(T_j,F)=\frac{1}{|F|}\#\{c\in G: T_jc\in \mathcal T \text{ and } T_jc\subset F\}|T_j|,$$
the density or the portion of tiles of $\mathcal{T}$ with shape $T_j$ completely contained in $F$.
Clearly, $\sum_{j=1}^l\rho_{\mathcal T}(T_j,F)\le 1$. But $\lim_{n\rightarrow\infty}\rho_{\mathcal T}(T_j,F_n)$ may not exist
for every F{\o}lner sequences $\{F_n\}$.

\begin{lemma}\label{tiling-quasi2}
  Let $\{F_n\}$ be any tempered F{\o}lner sequence in $G$. For any $K\in F(G)$ and $0<\e<\frac{1}{2}$, there exists a finite tiling $\mathcal T=\mathcal T(K,\e)$ of $G$ such that
  \begin{enumerate}
    \item $\mathcal T$ has shapes $T_1,T_2,\ldots,T_l$ and each shape is $(K,\e)$-invariant;
    \item for sufficiently large $n\in\N$, for $1\le j\le l$,
    there exists $\tilde F_n\subset F_n$ with $|\tilde F_n|>(1-\e)|F_n|$ such that
    $$\bigg|\frac{1}{|F_n|}\sum_{g\in F_n}1_{C_jg^{-1}\cap F_n}(h)-\frac{\rho_{\mathcal T}(T_j,F_n)}{|T_j|}\bigg|< \e\frac{\rho_{\mathcal T}(T_j,F_n)}{|T_j|},$$ for all $h\in \tilde F_n$,
    where $C_j=C(T_j)$ is the center of the shape $T_j$.
  \end{enumerate}
\end{lemma}
\begin{proof}
  By Theorem \ref{tiling2}, there exists a finite tiling $\mathcal T'$ with finite many shapes each of which is $(K,\e)$-invariant. Let $(X_{\mathcal T'},G)$ be the associated subshift. By choosing a minimal point from $(X_{\mathcal T'},G)$, we can make a new finite tiling and still denote it and the associated subshift by $\mathcal T'$ and $(X_{\mathcal T'},G)$, respectively.

  Let $T_1,T_2,\ldots,T_l$ be the shapes of $\mathcal T'$ and $\mu$ be a $G$-invariant ergodic measure of $(X_{\mathcal T'},G)$. For $x=(x_g)_{g\in G}\in X_{\mathcal T'}$ and each $j=1,2,\ldots,l$, define
  \begin{align*}
    f_j(x)=\begin{cases}
      1, \text{ if } x_{e_G}=T_j;\\
      0, \text{ otherwise. }
    \end{cases}
  \end{align*}

  By the pointwise ergodic theorem, for $\mu$-a.e. $x\in X_{\mathcal T'}$,
  $$\lim_{n\rightarrow \infty}\frac{1}{|F_n|}\sum_{g\in F_n}f_j(gx)=\int f_j(x)d\mu:=t_j.$$
  Note that since $(X_{\mathcal T'},G)$ is minimal, $t_j>0$. Hence for sufficiently large $N_0\in \N$,
  $$\mu\bigg(\big \{x\in X_{\mathcal T'}: \text{ for any }n>N_0, \big|\frac{1}{|F_n|}\sum_{g\in F_n}f_j(gx)-t_j\big|<\frac{\e}{6}t_j\big \}\bigg)>1-\e.$$
  Denote by $X_0$ the set in the left-hand side of the above inequality. Applying the pointwise ergodic theorem again, there exists $N_1>N_0$
  such that for any $n>N_1$, it holds that
  $$\mu\bigg(\big\{x\in X_{\mathcal T'}:\text{ for any }n>N_1, \frac{1}{|F_n|}\sum_{g\in F_n}1_{X_0}(gx)>1-\e\big\}\bigg)>1-\e.$$
  Denote by $X_1$ the set in the left-hand side of the above inequality.
  Now we choose $x\in X_0\cap X_1$ and let $\mathcal T$ be the finite tiling generated by $x$. Since $(X_{\mathcal T'},G)$ is minimal, $\mathcal T$ still has the same shapes $T_1,T_2,\ldots,T_l$ as $\mathcal T'$.

  Since the tiling $\mathcal T$ is generated by $x$, there exists $N_2>N_1$ such that whenever $n>N_2$, it holds that
  \begin{align}\label{ineq-F0}
    \bigg|\frac{1}{|F_n|}\sum_{g\in F_n}f_j(gx)-\frac{\rho_{\mathcal T}(T_j,F_n)}{|T_j|}\bigg|<\frac{\e}{6}t_j.
  \end{align}

  Now let $n>N_2$. Since $x\in X_0\cap X_1$, we have that
  \begin{align}\label{ineq-F1}
    \bigg|\frac{1}{|F_n|}\sum_{g\in F_n}f_j(gx)-t_j\bigg|<\frac{\e}{6}t_j
  \end{align}
  and
  \begin{align}\label{ineq-F2}
    \frac{1}{|F_n|}\sum_{g\in F_n}1_{X_0}(gx)>1-\e.
  \end{align}

  Joint \eqref{ineq-F0} and \eqref{ineq-F1} together, it holds that
  \begin{align}\label{ineq-F3}
    \frac{t_j}{2}<\frac{\rho_{\mathcal T}(T_j,F_n)}{|T_j|}<(1+\frac{\e}{6})t_j.
  \end{align}

  Let $\tilde F_n=\{h\in F_n: hx\in X_0\}$. Then by \eqref{ineq-F2}, $|\tilde F_n|> (1-\e)|F_n|$.

  For each $h\in\tilde F_n$, since $hx\in X_0$, it holds that
   $$\bigg|\frac{1}{|F_n|}\#\{g\in F_n: (hx)_g=x_{gh}=T_j\}-t_j\bigg|=\bigg |\frac{1}{|F_n|}\sum_{g\in F_n}f_j(ghx)-t_j\bigg |<\frac{\e}{6}t_j.$$
  Note that $x_{gh}=T_j$ if and only if $gh\in C_j$, i.e. $h\in C_jg^{-1}$.
  Hence
  $$\bigg|\frac{1}{|F_n|}\sum_{g\in F_n}1_{C_jg^{-1}\cap F_n}(h)-t_j\bigg|< \frac{\e}{6}t_j,$$ for all $h\in \tilde F_n$.

  Then whenever $n>N_2$, we have
  $$\bigg|\frac{1}{|F_n|}\sum_{g\in F_n}1_{C_jg^{-1}\cap F_n}(h)-\frac{\rho_{\mathcal T}(T_j,F_n)}{|T_j|}\bigg|< \frac{\e}{3}t_j<\e\frac{\rho_{\mathcal T}(T_j,F_n)}{|T_j|},$$ for all $h\in \tilde F_n$.

\end{proof}

\begin{remark}
\begin{enumerate}
  \item Here we need the F{\o}lner sequence $\{F_n\}$ to be tempered since we apply the pointwise ergodic theorem.
  \item From the proof of Lemma \ref{tiling-quasi2}, we can see that the shapes $T_1,T_2,\ldots,T_l$ do not depend on the given F{\o}lner sequence $\{F_n\}$,
  although the tiling $\mathcal T$ itself does depend on $\{F_n\}$.
\end{enumerate}

\end{remark}

With the help of Lemma \ref{tiling-quasi2}, we can construct a specific F{\o}lner sequence of $G$, which plays a crucial role for proving the variational principles.

\begin{lemma}\label{lemma-folner}
Let $\{H_n\}$ be any tempered F{\o}lner sequence of $G$. There exists a F{\o}lner sequence $\{F_n\}$ of $G$ (independent on $\{H_n\}$), such that
for any $e_G\in K\in F(G)$ and $0<\e<\frac{1}{2}$, there is a finite tiling $\mathcal{T}$ of $G$ satisfying the following:
\begin{enumerate}
  \item $\mathcal{T}$ has shapes $\{F_{m_1},\ldots,F_{m_l}\}$ consisted with F{\o}lner sets in $\{F_n\}$ each of which
  is $(K,\e)$-invariant;
  \item let $C_j$ be the center of the shape $F_{m_j}$ for each $1\le j\le l$, then the family of sets $\{C_jg^{-1}\}_{g\in H_n}$ covers a subset $\tilde{H}_n\subset H_n$ with
  $|\tilde{H}_n|>(1-\e)|H_n|$ at most $(1+\e)\rho_{\mathcal T}(F_{m_j},H_n)\frac{|H_n|}{|F_{m_j}|}$-many times, whenever $n$ is sufficiently large.
\end{enumerate}
\end{lemma}
\begin{proof}
  Let $\{\e_n\}$ be a sequence of real numbers decreasing to $0$
and let $\{K_n\}$ be a sequence of finite subsets of $G$ such that
\begin{enumerate}
  \item $\{e_G\}\subset K_1\subset K_2\subset\cdots$ and $\lim_{n\rightarrow\infty}K_n=G$;
  \item $K_n$ becomes more and more invariant as $n\rightarrow\infty$ (in fact $\{K_n\}$ is also a F{\o}lner sequence).
\end{enumerate}
Then we collect the shapes of tiling $\tilde{\mathcal T}(K_n,\e_n)$ associated with each pair $(K_n,\e_n)$ due to Theorem \ref{tiling2}
to form a sequence of finite subsets of $G$ and denote this sequence by $\{F_n\}$. Since the shapes become more and more invariant as $n\rightarrow\infty$,
$\{F_n\}$ is a F{\o}lner sequence of $G$.

For any $K\in F(G)$ and $\varepsilon>0$, let $K_n\supset K$ and $\e_n<\e$. We then take the finite tiling $\mathcal T'$ to be $\mathcal T'=\mathcal T(K_n,\e_n)$ as in Lemma \ref{tiling-quasi2}. Then every shape of $\mathcal T'$ is taken from the F{\o}lner sequence $\{F_n\}$ and $(K_n,\e_n)$-invariant (hence $(K,\e)$-invariant).

Moreover, by the same argument as in the proof of Lemma \ref{tiling-quasi2}, we can use the tiling $\mathcal T'$ to form the required tiling $\mathcal T$.  Let $F_{m_1},\ldots,F_{m_l}$ be the shapes of $\mathcal T$. Then for any sufficiently large $n\in\N$, for $1\le j\le l$, there exists $\tilde H_n\subset H_n$ with $|\tilde H_n|>(1-\e)|H_n|$
such that
    $$\bigg|\frac{1}{|H_n|}\sum_{g\in H_n}1_{C_jg^{-1}\cap H_n}(h)-\frac{\rho_{\mathcal T}(F_{m_j},H_n)}{|F_{m_j}|}\bigg|< \e\frac{\rho_{\mathcal T}(F_{m_j},H_n)}{|F_{m_j}|}, \text{ for any }h\in \tilde H_n.$$
Hence for any $h\in\tilde H_n$,
$$\frac{1}{|H_n|}\sum_{g\in H_n}1_{C_jg^{-1}\cap H_n}(h)<(1+\e)\frac{\rho_{\mathcal T}(F_{m_j},H_n)}{|F_{m_j}|}.$$
This shows that the set $\tilde{H}_n$ is covered by the family of sets $\{C_jg^{-1}\}_{g\in H_n}$ at most $(1+\e)\rho_{\mathcal T}(F_{m_j},H_n)\frac{|H_n|}{|F_{m_j}|}$-many times.
\end{proof}

\begin{remark}
  Since the construction of the F{\o}lner sequence $\{F_n\}$ is independent on the given tempered F{\o}lner sequence $\{H_n\}$,
  we can make $\{H_n\}$ to be a tempered subsequence of $\{F_n\}$. It would be more convenient if we can choose
  $\{H_n\}$ just to be the whole $\{F_n\}$, but we don't know whether we can make the whole F{\o}lner sequence $\{F_n\}$ tempered.
\end{remark}

\section{Mean dimension, mutual information and rate distortion function}
\subsection{Topological mean dimension and metric mean dimension}
Let $\X$ be a compact metrizable  space and $\alpha=\{U_1,U_2,\ldots, U_k\}$ be a finite open cover of $\X$. The {\it order}
of $\alpha$ is defined by
$${\rm ord}(\alpha)=\max_{x\in X}\sum_{i=1}^k1_{U_i}(x)-1.$$
Denote by $${\rm D}(\alpha)=\min_{\beta}{\rm ord}(\beta),$$
where $\beta$ is taken over all finite open covers of $\X$ with $\beta\succ\alpha$.

The {\it topological dimension} of $\X$ is then defined by
$$\dim \X=\sup_{\alpha} {\rm D}(\alpha),$$
where $\alpha$ runs over all finite open covers of $\X$.

Let $(\X,G)$ be a $G$-system, where $G$ is a countably infinite amenable group. For $F\in F(G)$ and a finite open cover $\alpha$ of $\X$, denote by $\alpha_F=\bigvee_{g\in F}g^{-1}\alpha$.
Then we can define
$${\rm D}(\alpha, G)=\lim_{n\rightarrow \infty}\frac{{\rm D}(\alpha_{F_n})}{|F_n|},$$
where $\{F_n\}$ is a F{\o}lner sequence of $G$. It is known that this limit exists and is independent on the choice of the F{\o}lner sequence.
The {\it mean topological dimension} ${\rm mdim}(\X,G)$ of $(\X,G)$ is defined by
$${\rm mdim}(\X,G)=\sup_{\alpha} {\rm D}(\alpha,G),$$
where $\alpha$ runs over all finite open covers of $\X$.

Let $(\mathcal{X},G)$ be a $G$-system with metric $d$. For $F\in F(G)$, define metrics $d_F$ and $\bar{d}_F$ on $\X$ by
$$d_F(x,y)=\max_{g\in F}d(gx,gy)$$
and
$$\bar{d}_F(x,y)=\frac{1}{|F|}\sum_{g\in F}d(gx,gy), x,y\in \X.$$
We note here that we also use $\bar d_F$ to denote the metric on $\X^F$ defined by
\begin{align}\label{metric}
  \bar{d}_F\big((x_g)_{g\in F},(y_g)_{g\in F}\big)=\frac{1}{|F|}\sum_{g\in F}d(x_g,y_g),
\end{align}
for $(x_g)_{g\in F},(y_g)_{g\in F}\in \X^F.$

For any $\e>0$, let $\#(\mathcal{X},d,\e)$ be the minimal cardinality of open cover $\mathcal{U}$ of
$\mathcal{X}$ with ${\rm diam }(\mathcal{U},d)<\e$. Then define
$$S(\X,G,d,\e)=\lim_{n\to\infty}\frac{1}{|F_n|}\log \#(\X,d_{F_n},\e).$$
This limit always exists and does not depend on the choice of the F{\o}lner sequence $\{F_n\}$.
Note that $h_{top}(\X,G)$, the topological entropy of the system $(\X,G)$, equals $\lim_{\e\to 0}S(\X,G,d,\e)$ for any metric $d$ which is compatible with the topology of $\X$.

The upper and lower {\it metric mean dimension} is then defined by
$${\rm \overline{mdim}_M}(\X,G,d)=\limsup_{\e\to 0}\frac{S(\X,G,d,\e)}{|\log \e|},$$
$${\rm \underline{mdim}_M}(\X,G,d)=\liminf_{\e\to 0}\frac{S(\X,G,d,\e)}{|\log \e|}.$$
When the limits agree, the common value is denoted by ${\rm mdim_M}(\X,G,d)$.

Replacing $d_F$ by $\bar{d}_F$ in the definition of $S(\X,G,d,\e)$, we can define for a F{\o}lner sequence $\{F_n\}$
$$\overline{S}(\X,\{F_n\},d,\e)=\limsup_{n\to\infty}\frac{1}{|F_n|}\log \#(\X,\bar{d}_{F_n},\e),$$
$$\underline{S}(\X,\{F_n\},d,\e)=\liminf_{n\to\infty}\frac{1}{|F_n|}\log \#(\X,\bar{d}_{F_n},\e).$$
It is easy to see that
$$\underline{S}(\X,\{F_n\},d,\e)\le\overline{S}(\X,\{F_n\},d,\e)\le S(\X,G,d,\e).$$

\begin{remark}
It was proved in Theorem 6.1 of \cite{LW} that if a function $f: F(G)\rightarrow \R$ satisfies
  \begin{enumerate}
    \item $f(F)\ge 0$ for any $F\in F(G)$ and $f(\emptyset)=0$,
    \item $f(E)\le f(F)$ for any $E,F\in F(G)$ with $E\subset F$,
    \item $f(Fg)=f(F)$ for any $F\in F(G)$ and $g\in G$,
    \item $f(E\cup F)\le f(E)+f(F)$ for any $E,F\in F(G)$ with $E\cap F=\emptyset$,
  \end{enumerate}
then $\frac{1}{|F|}f(F)$ converges to a limit as $F$ becomes more and more invariant. As a function of $F\in F(G)$, $\log \#(\X,\bar d_{F},\e)$ satisfies
{\rm (1),(3)} and {\rm (4)} but does not satisfy {\rm (2)}. Hence we can not derive $\underline{S}(\X,\{F_n\},d,\e)=\overline{S}(\X,\{F_n\},d,\e)$ from this theorem.
\end{remark}

\begin{proposition}\label{prop-S}
  Let $\{F_n\}$ and $\{H_n\}$ be any two F{\o}lner sequences and $\e>0$. Then
  $$\overline{S}(\X,\{H_n\},d,2\e)\le \underline{S}(\X,\{F_n\},d,\e).$$
\end{proposition}
\begin{proof}
  Let $\e>0$ be fixed. Passing to a subsequence of $\{F_n\}$ (we still denote it by $\{F_n\}$), assume that
  $$s:=\underline{S}(\X,\{F_n\},d,\e)= \lim_{n\to\infty}\frac{1}{|F_n|}\log \#(\X,\bar{d}_{F_n},\e).$$
  For any $\delta>0$, there exists $N\in\N$ such that for any $n>N$, $\#(\X,\bar{d}_{F_n},\e)\le e^{|F_n|(s+\delta)}$.
  For any $0<\eta<1/4$, by Theorem \ref{quasi-tiling}, there exist $N<n_1<n_2<\cdots<n_k$ such that
  $H_M$ can be $\eta$-quasi-tiled by $F_{n_1},F_{n_2},\cdots,F_{n_k}$, whenever $M$ is sufficiently large.
  Denote by $C_1,C_2,\cdots,C_k$ the tiling centers of this quasi-tiling.

  For each $1\le i\le k$, let $\mathcal{U}_i$ be an open cover of $\X$ with ${\rm diam }(\mathcal{U}_i,d_{F_{n_i}})<\e$ such that
  $\mathcal{U}_i$ has minimal cardinality $\#(\X,\bar{d}_{F_{n_i}},\e)$.

  Let $\mathcal {U}_0$ be a finite open cover of $\X$ with ${\rm diam }(\mathcal{U}_0,d_{e_G})<\e$ and assume $\#\mathcal{U}_0=L$. Now we construct an open cover $\mathcal U$ of $\X$ by
  $$\mathcal{U}=\bigg(\bigvee_{i=1}^k\bigvee_{c\in C_i}c^{-1}\mathcal {U}_i\bigg)\bigvee_{g\in H_M\setminus \cup_{i=1}^kF_{n_i}C_i}g^{-1}\mathcal{U}_0.$$

  Note that for any $g\in G$, $F\in F(G)$ and any finite open cover $\mathcal{V}$ of $\X$,
  $${\rm diam }(g^{-1}\mathcal{V},d_{Fg})={\rm diam }(\mathcal{V},d_{F}) \text{ and } \#g^{-1}\mathcal{V}=\#\mathcal{V}.$$

  Hence
  \begin{align*}
    {\rm diam }(\mathcal{U},d_{H_M})&\le \frac{1}{|H_M|}\bigg(\sum_{i=1}^k|C_i|\cdot|F_{n_i}|{\rm diam }(\mathcal{U}_i,d_{F_{n_i}})\\
    &\qquad\qquad\qquad\qquad\qquad\qquad+|H_M\setminus \cup_{i=1}^kF_{n_i}C_i|{\rm diam }(\mathcal{U}_0,d_{e_G})\bigg)\\
    &<\frac{1}{|H_M|}(\sum_{i=1}^k|C_i|\cdot|F_{n_i}|+|H_M\setminus \cup_{i=1}^kF_{n_i}C_i|)\e\\
    &<\big(\frac{1}{1-\eta}+\eta\big)\e \qquad \text{ (by \eqref{ineq-quasi-tiling})} \\
    &<2\e.
  \end{align*}
  Moreover,
  \begin{align*}
    \#\mathcal{U}&\le \bigg(\prod_{i=1}^k\prod_{c\in C_i}\#c^{-1}\mathcal {U}_i\bigg)\cdot\bigg(\prod_{g\in H_M\setminus \cup_{i=1}^kF_{n_i}C_i}\#g^{-1}\mathcal{U}_0\bigg)\\
    &= \bigg(\prod_{i=1}^k|C_i|e^{|F_{n_i}|(s+\delta)}\bigg)\cdot \bigg(L^{|H_M\setminus \cup_{i=1}^kF_{n_i}C_i|}\bigg)\\
    &\le \exp\bigg(|H_M|\big(\frac{1}{1-\eta}(s+\delta)+\eta\log L\bigg).
  \end{align*}
  Letting $\delta, \eta\rightarrow 0$, we have that
  \begin{align*}
    \overline{S}(\X,\{H_n\},d,2\e)\le \underline{S}(\X,\{F_n\},d,\e).
  \end{align*}
\end{proof}

As a corollary, we have
\begin{corollary}\label{coro-S}
  Let $\{F_n\}$ be a F{\o}lner sequence, then
  $$\limsup_{\e\to 0}\frac{\underline{S}(\X,\{F_n\},d,\e)}{|\log \e|}=\limsup_{\e\to 0}\frac{\overline{S}(\X,\{F_n\},d,\e)}{|\log \e|},$$
  $$\liminf_{\e\to 0}\frac{\underline{S}(\X,\{F_n\},d,\e)}{|\log \e|}=\liminf_{\e\to 0}\frac{\overline{S}(\X,\{F_n\},d,\e)}{|\log \e|}.$$
Moreover, the above limits do not depend on the choice of the F{\o}lner sequence $\{F_n\}$.
\end{corollary}

\begin{proposition}\label{mdim-equal}
  If $(\X,d)$ has tame growth of covering numbers, then for any F{\o}lner sequence $\{F_n\}$,
  $${\rm \overline{mdim}_M}(\X,G,d)=\limsup_{\e\to 0}\frac{\underline{S}(\X,\{F_n\},d,\e)}{|\log \e|}=\limsup_{\e\to 0}\frac{\overline{S}(\X,\{F_n\},d,\e)}{|\log \e|},$$
$${\rm \underline{mdim}_M}(\X,G,d)=\liminf_{\e\to 0}\frac{\underline{S}(\X,\{F_n\},d,\e)}{|\log \e|}=\liminf_{\e\to 0}\frac{\overline{S}(\X,\{F_n\},d,\e)}{|\log \e|}.$$
\end{proposition}
\begin{proof}
We only need to prove the case of ${\rm \overline{mdim}_M}(\X,G,d)$. The case of ${\rm \overline{mdim}_M}(\X,G,d)$ is similar.

Since $\underline{S}(\X,\{F_n\},d,\e)\le S(\X,G,d,\e)$, it obviously holds that
$${\rm \overline{mdim}_M}(\X,G,d)=\limsup_{\e\to 0}\frac{ S(\X,G,d,\e)}{|\log \e|}\ge\limsup_{\e\to 0}\frac{\overline{S}(\X,\{F_n\},d,\e)}{|\log \e|}.$$

Let $\mathcal W=\{W_1, \ldots, W_M\}$ be an open cover of $\X$ with ${\rm diam} (\mathcal W,d)<\e$ and $M=\# (\X,d,\e)$.
Respectively, for $F\in F(G)$, let $\mathcal U=\{U_1,\ldots,U_N\}$ be an open cover of $\X$ with ${\rm diam} (\mathcal U,\bar d_{F})<\e$ and $N=\#(\X,\bar d_F,\e)$.

Now for each $1\le i\le N$ choose a point $p_i\in U_i$. Then $\bar d_{F}(x,p_i)<\e$ for every $x\in U_i$. Hence for $L\ge 1$,
$$|\{g\in F: d(gx,gp_i)\ge L\e\}|<\frac{|F|}{L},$$
which follows that
$$U_i\subset \bigcup_{A\subset F \text{ with } |A|<\frac{|F|}{L}}B_{L\e}(p_i,d_{F\setminus A}).$$

For $A\subset F$, since $\bigvee_{g\in A}g^{-1}\mathcal{W}$ is a cover of $X$, it holds that
$$B_{L\e}(p_i,d_{F\setminus A})=\bigcup_{(m_g)_{g\in A}\in \{1,2,\cdots,M\}^A}\big (\cap_{g\in A}g^{-1}W_{m_g}\cap B_{L\e}(p_i,d_{F\setminus A})\big ).$$
Noticing that $${\rm diam }\big(\cap_{g\in A}g^{-1}W_{m_g}\cap B_{L\e}(p_i,d_{F\setminus A}),d_F\big)<2L\e,$$
we have for $A\subset F$ with $|A|<\frac{|F|}{L}$,
$$\#(B_{L\e}(p_i,d_{F\setminus A}),d_F,2L\e)\le M^{|A|}\le M^{\frac{|F|}{L}}.$$
Since there are $N$ choices of $U_i$ and no more than $2^{|F|}$ many choices of $A\subset F$, it holds that
$$\#(\X,d_F,2L\e)\le 2^{|F|}M^{\frac{|F|}{L}}N.$$
Thus
$$\frac{1}{|F|}\log \#(\X,d_F,2L\e)\le \log 2+\frac{1}{L}\log \#(\X,d,\e)+\frac{1}{|F|}\log \#(\X,\bar d_F,\e).$$

Now take $0<\delta<1$ and $L=(1/\e)^\delta$, we have
$$\frac{1}{|F|}\log \#(\X,d_F,2\e^{1-\delta})\le \log 2+\e^{\delta}\log \#(\X,d,\e)+\frac{1}{|F|}\log \#(\X,\bar d_F,\e).$$
Letting $F=F_n$ with $n\rightarrow \infty$ in any F{\o}lner sequence $\{F_n\}$,
$$S(\X,G,d,2\e^{1-\delta})\le \log 2+\e^{\delta}\log \#(\X,d,\e)+\overline{S}(\X,\{F_n\},d,\e).$$
Using the condition of the tame growth of covering numbers and then letting $\delta\rightarrow 0$, it follows that
$${\rm \overline{mdim}_M}(\X,G,d)\le\limsup_{\e\to 0}\frac{\overline{S}(\X,\{F_n\},d,\e)}{|\log \e|}.$$
Applying Corollary \ref{coro-S}, we finish the proof of the proposition.
\end{proof}

\subsection{Mutual information}\

Mutual information is a fundamental and important concept in information theory via entropy.
In this subsection we will introduce its definition and collect some of its basic properties from \cite{LT2}.

Let $(\Omega,\mathbb P)$ be a probability space, $\X,\Y$ be two measurable spaces and $X:\Omega\rightarrow \X$ and $Y:\Omega\rightarrow \Y$ be two measurable maps.
$I(X;Y)$, the {\it mutual information} of $X$ and $Y$ is defined by the following:
\begin{align}\label{def-mutual}
  I(X;Y):=\sup_{\mathcal{P},\mathcal{Q}}\sum_{P\in\mathcal{P},Q\in \mathcal {Q}}\p((X,Y)\in P\times Q)\log \frac{\p((X,Y)\in P\times Q)}{\p(X\in P)\p(Y\in Q)},
\end{align}
where $\mathcal{P}$ and $\mathcal{Q}$ run over all finite measurable partitions of $\X$ and $\Y$ respectively. With the convention we set that $0\log \frac{0}{a}=0$ for all $a\ge 0$.

It is easy to see that $I(X;Y)=I(Y;X)\ge 0$ for any measurable maps $X$ and $Y$.

The mutual information has the following properties.

\begin{proposition}[\cite{LT2}]\label{prop-mutual1}
  Let $\X,\Y,\z$ be measurable spaces , $X,Y,Z$ be measurable maps from $\Omega$ to $\X,\Y,\z$ respectively, and $f:\Y\rightarrow \z$ be a measurable map.
  \begin{enumerate}
    \item (Data-processing inequality).
  $$I(X;f(Y))\le I(X;Y).$$
    \item $I(X;Y)=H(X)-H(X|Y)=H(X)+H(Y)-H(X,Y)$.

     If in addition $\X,\Y$ and $\z$ are finite sets, then the following {\rm (3)-(6)} holds.

  \item Let $(X_n,Y_n):\Omega\rightarrow \X\times \Y$ be a sequence of measurable maps converging to $(X,Y)$ in law, then $I(X_n;Y_n)$ converges to $I(X;Y)$.

  \item (Fano's inequality). Let $P_e=\p(X\neq f(Y))$, then
  $$H(X|Y)\le H(P_e)+P_e\log |\X|.$$

  \item (Subadditivity). If $X$ and $Z$ are conditionally independent given $Y$, i.e. for every $y\in \Y$ with $\p(Y=y)\neq 0$ and for every $x\in \X$ and $y\in \Y$,
  $$\p(X=x,Z=z|Y=y)=\p(X=x|Y=y)\p(Z=z|Y=y),$$
  then $$I(Y;X,Z)\le I(Y;X)+I(Y;Z).$$

  \item (Superadditivity). If $X$ and $Z$ are independent, then
  $$I(Y;X,Z)\ge I(Y;X)+I(Y;Z).$$
  \end{enumerate}
\end{proposition}

Let $(X,d)$ be a compact metric space and $\e>0$. A subset $S\subset \X$ is said to be $\e$-separated if $d(x,y)\ge\e$ for any two distinct points $x,y\in S$.
The following lemma is Corollary 16 of \cite{LT2}, which is a corollary of Fano's inequality.

\begin{lemma}\label{lemma-Fano}
  Let $(X,d)$ be a compact metric space. Let $\e>0$ and $D>2$. Suppose $S\subset \X$ is a $2D\e$-separated set. Let $X$ and $Y$ be measurable maps from $\Omega$
  to $\X$ such that $X$ is uniformly distributed over $S$ and $\E\big(d(X,Y)\big)<\e$. Then
  \begin{align*}
    I(X;Y)\ge (1-\frac{1}{D})\log |S|-H(\frac{1}{D}).
  \end{align*}
\end{lemma}

Let $\X,\Y$ be finite and let $X$ and $Y$ be measurable maps from $\Omega$ to $\X$ and $\Y$ respectively.
Let $$\mu(x)=\p(X=x), \nu(y|x)=\p(Y=y|X=x),$$
then $\mu(x)\nu(y|x)$ determines the distribution of $(X,Y)$ and hence the mutual information $I(X;Y)$. So sometimes we use $I(\mu,\nu)$ to instead $I(X;Y)$.
The following lemma (\cite[Lemma 2.10]{LT2}) shows the concavity and convexity of mutual information.

\begin{lemma}\label{prop-mutual2}\

\begin{enumerate}
  \item Suppose for each $x\in \X$ we are given a probability mass function $\nu(\cdot|x)$ on $\Y$. Let $\mu_1$ and $\mu_2$ be two probability mass function on $\X$. Then
  $$I((1-t)\mu_1+t\mu_2,\nu)\ge (1-t)I(\mu_1,\nu)+tI(\mu_2,\nu)\qquad (0\le t\le 1).$$

  \item Suppose for each $x\in \X$ we are given two probability mass functions $\nu_1(\cdot|x)$ and $\nu_2(\cdot|x)$ on $\Y$. Let $\mu$ be a probability mass function on $\X$. Then
  $$I(\mu,(1-t)\nu_1+t\nu_2)\le (1-t)I(\mu,\nu_1)+tI(\mu,\nu_2)\qquad (0\le t\le 1).$$
\end{enumerate}
\end{lemma}

\subsection{Rate distortion functions}\

Now we introduce rate distortion functions for dynamical systems.

Let $(\mathcal{X},G)$ be a $G$-system with metric $d$, where $G$ is a countably infinite amenable group. Denote by $M(\X,G)$ the collection of $G$-invariant probability measures of $\X$.

 Let $\e>0$ and $\mu\in M(\X,G)$. For $F\in F(G)$, let $X: \Omega\rightarrow \X$ and $Y_g: \Omega\rightarrow \X, g\in F$ be random variables defined on some probability space $(\Omega,\p)$. Assume the law of $X$ is given by $\mu$. We say $X$ and $Y=(Y_g)_{g\in F}$ are {\it $(F,\e)$-close (or $(F,\e)_{L^1}$-close)} if
  \begin{align}\label{condition-rate}
    \E\bigg(\frac{1}{|F|}\sum_{g\in F}d(gX,Y_g)\bigg)<\e.
  \end{align}
\eqref{condition-rate} is also called the ($L^1$) distortion condition.

There are also $L^\infty$ and $L^p~(p\ge 1)$ versions of distortion conditions and $L^\infty$ and $L^p~(p\ge 1)$ rate distortion functions. We will consider them in section 5.

Denote by $$R_{\mu}(\e,F)=\inf_{X,Y \text{ are }(F,\e)\text{-close}}I(X;Y).$$

\begin{remark}\label{remark-Yfinite}\
   Similar to Remark 14 of \cite{LT2}, in the definition of $L^1$ rate distortion functions,
   the random variable $Y$ can be assumed to take only finitely many values.
\end{remark}

Let $\{F_n\}$ be a F{\o}lner sequence in $G$.
The upper and lower {\it ($L^1$) rate distortion functions} are then defined
by $$\overline{R}_{\mu}(\{F_n\},\e)=\limsup_{n\rightarrow \infty}\frac{R_{\mu}(\e,F_n)}{|F_n|},$$
$$\underline{R}_{\mu}(\{F_n\},\e)=\liminf_{n\rightarrow \infty}\frac{R_{\mu}(\e,F_n)}{|F_n|}.$$
Using quasi-tiling technique, similar to the proof of Proposition \ref{prop-S}, we have

\begin{proposition}\label{prop-R}
  Let $\{F_n\}$ and $\{H_n\}$ be any two F{\o}lner sequences and $\e>0$. Then
  $$\overline{R}_{\mu}(\{H_n\},2\e)\le \underline{R}_{\mu}(\{F_n\},\e).$$
\end{proposition}
\begin{proof}
  Let $\e>0$ be fixed. Passing to a subsequence of $\{F_n\}$ (we still denote it by $\{F_n\}$), assume that
  $$r:=\underline{R}_{\mu}(\{F_n\},\e)= \lim_{n\rightarrow \infty}\frac{R_{\mu}(\e,F_n)}{|F_n|}.$$
  If $r=+\infty$, then there is nothing to prove. So we assume $r<+\infty$.
  For any $\delta>0$, there exists $N\in\N$ such that for any $n>N$, $R_{\mu}(\e,F_n)\le |F_n|(r+\delta)$.
  For any $0<\eta<\min\{1/4,\frac{\delta}{2r}, \frac{\e}{8{\rm diam}(\X,d)}\}$, by Theorem \ref{quasi-tiling}, there exist $N<n_1<n_2<\cdots<n_k$ such that
  $H_M$ can be $\eta$-quasi-tiled by $F_{n_1},F_{n_2},\cdots,F_{n_k}$, whenever $M$ is sufficiently large.
  Denote by $C_1,C_2,\cdots,C_k$ the tiling centers of this quasi-tiling. For each set $F_{n_j}c_j$ ($1\le j\le k, c_j\in C_j$),
  we take $T_{j,c_j}\subset F_{n_j}c_j$ with $|T_{j,c_j}|\ge (1-\eta)|F_{n_j}|$ form a disjoint collection associated to this quasi-tiling,
  i.e. $\bigsqcup_{1\le j\le k, c_j\in C_j}T_{j,c_j}\subset H_M$. Let $R:=H_M\setminus \bigsqcup_{1\le j\le k, c_j\in C_j}T_{j,c_j}$. Then
  \begin{align*}
    |R|&=|H_M|-\sum_{1\le j\le k, c_j\in C_j}|T_{j,c_j}| \le |H_M|-\sum_{1\le j\le k}(1-\eta)|F_{n_j}|\cdot |C_j|\nonumber \\
       &\le |H_M|-(1-\eta)^2|H_M| \qquad\qquad \text{(by \eqref{ineq-quasi-tiling})}\nonumber\\
       &< 2\eta|H_M|.
  \end{align*}

  For each $1\le j\le k$, let $X^{(j)}: \Omega\rightarrow \X$ and $Y_g^{(j)}: \Omega\rightarrow \X, g\in F_{n_j}$ be random variables defined on the probability space $(\Omega,\p)$ such that
  \begin{enumerate}
    \item ${\rm Law} \big(X^{(j)}\big)=\mu$ and $X^{(j)}$ and $Y^{(j)}=(Y_g^{(j)})_{g\in F_{n_j}}$ are $(F_{n_j},\e)$-close;
    \item $I(X^{(j)}; Y^{(j)})\le R_{\mu}(\e,F_{n_j})+\delta \le |F_{n_j}|(r+2\delta)$.
  \end{enumerate}
  By Remark \ref{remark-Yfinite}, we can make the random variables $Y^{(j)}$'s take finitely many values. Assume the distribution of $Y^{j}$ is supported
  on a finite subset $\Y^{(j)}\subset \X^{F_{n_j}}$, $1\le j\le k$.

For a subset $F\subset F_{n_j}$, denote by $Y^{(j)}_F=(Y_g^{(j)})_{g\in F}$ and $\Y^{(j)}_F={\rm Proj}_F\Y^{(j)}\subset \X^{F}$, the restriction of $Y^{(j)}$ and $\Y^{(j)}$ to $F$-coordinates, resectively.
Similarly, for $y=(y_g)_{g\in H_M}\in \X^{H_M}$ and $F\subset H_M$, denote by $y_F=(y_g)_{g\in F}$.

For each $n>0$, let $\mathcal P_n=\{P_1,\ldots,P_M\}$ be a measurable partition of $\X$ with ${\rm diam}(P_m,d)<\frac{1}{n}$ for each $1\le m\le M$. Moreover, we assume the sequence of partitions $\{\mathcal P_n\}_{n=1}^\infty$ is increased and $\vee_{n=1}^\infty\mathcal P_n$ generates the whole Borel-$\sigma$ algebra of $\X$. For each $\mathcal P_n$, assign each $P_m$ a point $p_m\in P_m$ and set $A=\{p_1,\ldots,p_M\}$. Denote by $\mathcal P_n(x)=p_m$ for $x\in P_m$. The random variable $\mathcal P_n(X^{(j)})$ takes values on $A$ and ${\rm Law} \big(\mathcal P_n(X^{(j)})\big)$ is determined by the push-forward measure $\mathcal P_{n\,*}\mu$ which is given by
\begin{align*}
  \mathcal P_{n\,*}\mu(p_m)=\p(X^{(j)}\in P_m)=\mu(P_m), \text{ for } 1\le m\le M.
\end{align*}
Let $\mathcal P_n^{F}(x)=\big(\mathcal P_n(gx)\big)_{g\in F}$ for $F\in F(G)$. Then the random variable $\mathcal P_n^F(X)$ takes values on $A^F\subset \X^F$ and the push-forward measure $\mathcal P_{n\,*}^F\mu$ is given by a similar may:
\begin{align*}
  \mathcal P_{n\,*}^F\mu(x)=\p\Big(gX^{(j)}\in \mathcal{P}(x_g): g\in F\Big)=\mu\Big(\cap_{g\in F}\mathcal{P}(x_g)\Big), \text{ for } x=(x_g)_{g\in F}\in A^F.
\end{align*}

\if and ${\rm Law} \big(\mathcal P_n^F(X)\big)$ is determined by the push-forward measure $\mathcal P_{n\,*}\mu$ which is given by
\begin{align*}
  \mathcal P_{n\,*}\mu(p_m)=\p(X\in P_m)=\mu(P_m), \text{ for } 1\le m\le M.
\end{align*}
\fi

Define for each pair $(j,c_j)$ the conditional probability mass function $\nu^{(j,c_j)}(y|x)$ by
\begin{align*}
  \nu^{(j,c_j)}(y|x)=\p\Big(Y_{T_{j,c_j}c_j^{-1}}^{(j)}=y|\mathcal P_n^{T_{j,c_j}c_j^{-1}}(X^{(j)})=x\Big), 
\end{align*}
where $y=(y_g)_{g\in T_{j,c_j}c_j^{-1}}\in \Y_{T_{j,c_j}c_j^{-1}}^{(j)}$ and $x=(x_g)_{g\in {T_{j,c_j}c_j^{-1}}}\in A^{T_{j,c_j}c_j^{-1}}$.
Then
\begin{align*}
  I( \mathcal P_{n\,*}^{T_{j,c_j}c_j^{-1}}\mu, \nu^{(j,c_j)})&=I\Big(\mathcal P_n^{T_{j,c_j}c_j^{-1}}(X^{(j)});Y_{T_{j,c_j}c_j^{-1}}^{(j)}\Big)\le I(X^{(j)}; Y^{(j)})\le |F_{n_j}|(r+2\delta),
\end{align*}
by the Data-processing inequality (Proposition \ref{prop-mutual1} (1)). 

Let $a\in \X$ be fixed. Now for $x=(x_g)_{g\in H_M}\in A^{H_M}$ we construct a probability mass function $\sigma_n(\cdot|x)$ on a finite subset $\Y_{H_M}$ of $\X^{H_M}$ as follows:
\begin{align*}
  \sigma_n(y|x):=\prod_{1\le j\le k, c_j\in C_j}\nu^{(j,c_j)}(y_{T_{j,c_j}}|x_{T_{j,c_j}})\cdot \prod_{g\in R}\delta_a(y_g).
\end{align*}
Here $y_{T_{j,c_j}}$ and $x_{T_{j,c_j}}$ should be understood as points in $\Y_{T_{j,c_j}c_j^{-1}}^{(j)}$ and $A^{T_{j,c_j}c_j^{-1}}$ by the $c_j^{-1}$-translation on coordinates and $\delta_a$ is the point probability mass function on $a$.

Hence we have
\begin{align*}
  I(\mathcal P_{n\,*}^{H_M}\mu,\sigma_n)&\le \sum_{1\le j\le k, c_j\in C_j}I\big(\mathcal P_{n\,*}^{T_{j,c_j}}\mu, \nu^{(j,c_j)}\big) \\
  &\qquad\text{(by Proposition \ref{prop-mutual1} (5), the subadditivity of mutual information)}\\
  &\le \sum_{j=1}^k|C_j|\cdot|F_{n_j}|\big(r+2\delta\big)\\
  &\le \frac{1}{1-\eta}|H_M|\big(r+2\delta\big)\qquad \text{(by \eqref{ineq-quasi-tiling})}\\
  &\le |H_M|\big(1+2\eta\big)\big(r+2\delta\big)\le |H_M|\big(r+4\delta\big).
\end{align*}

Denote by $\E_{\mathcal P_{n\,*}^{H_M}\mu,\sigma_n}$ the expectation with respect to the measure $\mathcal P_{n\,*}^{H_M}\mu(x)\sigma_n(y|x)$.
We have
\begin{align*}
  &\qquad\E_{\mathcal P_{n\,*}^{H_M}\mu,\sigma_n}\bigg( \bar d_{H_M}\big((gx)_{g\in H_M},y\big)\bigg)\\
  &\le \frac{1}{|H_M|}\sum_{1\le j\le k, c_j\in C_j}|T_{j,c_j}|\E_{\mathcal P_{n\,*}^{T_{j,c_j}c_j^{-1}}\mu, \nu^{(j,c_j)}}\bigg(\bar d_{T_{j,c_j}}\big((gx)_{g\in T_{j,c_j}},y_{T_{j,c_j}}\big)\bigg)\\
  &\qquad\qquad\qquad +\frac{1}{|H_M|}\sum_{g\in R}\E_{g_*\mathcal P_{n\,*}\mu,\delta_a}\bigg(d(gx,y_g)\bigg)\\
  &\le \frac{1}{|H_M|}\sum_{1\le j\le k, c_j\in C_j}|T_{j,c_j}|\E\bigg(\bar d_{T_{j,c_j}c_j^{-1}}\big(\mathcal P_n^{T_{j,c_j}c_j^{-1}}(X^{(j)}),Y^{(j)}_{T_{j,c_j}c_j^{-1}}\big)\bigg)+ |R|{\rm diam}(\X,d)\\
  &\le \frac{1}{|H_M|}\sum_{1\le j\le k, c_j\in C_j}|T_{j,c_j}|\Bigg(\E\bigg(\bar d_{T_{j,c_j}c_j^{-1}}\big((gX^{(j)})_{g\in T_{j,c_j}c_j^{-1}},Y^{(j)}_{T_{j,c_j}c_j^{-1}}\big)\bigg)\\
  &\qquad\qquad\qquad +\E\bigg(\bar d_{T_{j,c_j}c_j^{-1}}\big(\mathcal P_n^{T_{j,c_j}c_j^{-1}}(X^{(j)}),(gX^{(j)})_{g\in T_{j,c_j}c_j^{-1}}\big)\bigg)\Bigg)+ 2\eta{\rm diam}(\X,d)\\
  &\le \frac{1}{|H_M|}\bigg(\sum_{1\le j\le k}|C_j|\cdot|T_{j,c_j}|\Big(\e+\frac{1}{n}\Big)\bigg)+\frac{\e}{4}.
\end{align*}
Hence when when $n$ is sufficiently large, 
\begin{align}\label{ineq-E}
  \E_{\mathcal P_{n\,*}^{H_M}\mu,\sigma_n}\bigg( \bar d_{H_M}\big((gx)_{g\in H_M},y\big)\bigg)<\frac{3}{2}\e.
\end{align}

Let $\tau$ be a limit point of the sequence of probability measures $\{\mathcal P_{n\,*}^{H_M}\mu(x)\sigma_n(y|x)\}_{n=1}^\infty$ (passing to a subsequence, we still assume $\mathcal P^{H_M}_{n\,*}\mu(x)\sigma_n(y|x)\rightarrow \tau$). We note that $\tau$ is supported on $A^{H_M}\times \Y_{H_M}\subset \X^{H_M}\times\X^{H_M}$. Since $\mathcal P_n$ becomes finer and finer as $n\rightarrow\infty$, the probability measure $\mathcal P_{n\,*}\mu$ converges to $\mu$ in the weak* topology. Hence the projection of the first marginal of $\tau$ (say $\tilde \mu$) to any coordinate is $\mu$. Moreover, $\tilde \mu$ is supported
on $\{(gx)_{g\in H_M}: x\in {\rm supp}\,\mu\}$.

For any random variables $\tilde X=(\tilde X_g)_{g\in H_M}$ and $Y=(Y_g)_{g\in H_M}$ with ${\rm Law}\big(\tilde X,Y\big)$ obeying $\tau$, from the definition of the mutual information \eqref{def-mutual},
$I(\mathcal P_{n\,*}^{H_M}\mu,\sigma_n)\rightarrow I(\tilde X;Y)$. Let $X=g^{-1}\tilde X_g$ for some $g\in H_M$ (in fact, we can view $\tilde X$ as $(gX)_{g\in H_M}$). Then it holds that for sufficiently large $n$,
\begin{align*}
  I(X;Y)\le I(\tilde X;Y)\le I(\mathcal P_{n\,*}\mu,\sigma_n)+\delta\le |H_M|\big(r+5\delta\big).
\end{align*}

 Then by \eqref{ineq-E}, for sufficiently large $n$,
  \begin{align*}
    \E\bigg(\frac{1}{|H_M|}\sum_{g\in H_M}d(gX,Y_g)\bigg)&=\int_{\X^{H_M}\times \X^{H_M}}\frac{1}{|H_M|}\sum_{g\in H_M}d(gx,y_g)d\tau\big((gx)_{g\in H_M},(y_g)_{g\in H_M}\big)\\
    &\le
    \E_{\mathcal P_{n\,*}^{H_M}\mu,\sigma_n}\bigg( \bar d_{H_M}\big((gx)_{g\in H_M},y\big)\bigg)+\frac{\e}{2}< 2\e.
  \end{align*}
  
  Letting $\delta\rightarrow 0$, we have that
  \begin{align*}
    \overline{S}(\X,\{H_n\},d,2\e)\le \underline{S}(\X,\{F_n\},d,\e).
  \end{align*}
\end{proof}

\section{$L^1$ variational principle}
In this section, we will prove the following $L^1$ variational principle between metric mean dimension and the $L^1$ rate distortion function.

\begin{theorem}\label{theorem-L1}
  If $(\X,d)$ has tame growth of covering numbers, then
\begin{align*}
  {\rm \overline{mdim}_M}(\X,G,d)&=\limsup_{\e\to 0}\frac{\sup_{\mu\in M(\X,G)}\overline{R}_{\mu}(\{F_n\},\e)}{|\log \e|}\\
  &=\limsup_{\e\to 0}\frac{\sup_{\mu\in M(\X,G)}\underline{R}_{\mu}(\{F_n\},\e)}{|\log \e|}
\end{align*}
and
\begin{align*}
  {\rm \underline{mdim}_M}(\X,G,d)&=\liminf_{\e\to 0}\frac{\sup_{\mu\in M(\X,G)}\overline{R}_{\mu}(\{F_n\},\e)}{|\log \e|}\\
  &=\liminf_{\e\to 0}\frac{\sup_{\mu\in M(\X,G)}\underline{R}_{\mu}(\{F_n\},\e)}{|\log \e|},
\end{align*}
where $\{F_n\}$ is any F{\o}lner sequence.
\end{theorem}
\subsection{The lower bound}
\begin{lemma}\label{lemma-lower}
  For $\e>0$, $\mu\in M(\X,G)$ and a F{\o}lner sequence $\{F_n\}$, we have
  $$\overline{R}_{\mu}(\{F_n\},\e)\le \overline{S}(\X,\{F_n\},d,\e),$$
  $$\underline{R}_{\mu}(\{F_n\},\e)\le \underline{S}(\X,\{F_n\},d,\e).$$
\end{lemma}
\begin{proof}
For $n>0$, denote by $M=\#(\X,\bar d_{F_n},\e)$ and
let $\{U_1,\ldots,U_M\}$ be an open cover of $\X$ with ${\rm diam}(U_m,\bar d_{F_n})<\e$ for each $1\le m\le M$.
Choose a point $p_m\in U_m$ for each $m$. For any $x\in \X$, let $m$ be the smallest number satisfying $x\in U_m$. Then by setting $f(x)=p_m$ we can define a map $f:\X\rightarrow \{p_1,\ldots,p_M\}$ and hence $\bar d_{F_n}(x,f(x))<\e$. Let $X$ be a random variable obeying $\mu$ and let $Y=(gf(X))_{g\in F_n}$.
Then
$$\E\bigg(\frac{1}{|F_n|}\sum_{g\in F_n}d(gX,gf(X))\bigg)=\E~\bar d_{F_n}(X,f(X))<\e.$$
Hence $$I(X;Y)\le H(Y)\le \log M=\log\#(\X,\bar d_{F_n},\e).$$
Dividing by $|F_n|$ and taking limsup and liminf for $n\rightarrow\infty$, we complete the proof of the lemma.
\end{proof}

Since $\overline {S}(\X,\{F_n\},d,\e)\le S(\X,G,d,\e)$, together with Proposition \ref{prop-R}, we have
\begin{proposition}\label{prop-lower}
\begin{align*}
  {\rm \overline{mdim}_M}(\X,G,d)&\ge\limsup_{\e\to 0}\frac{\sup_{\mu\in M(\X,G)}\overline{R}_{\mu}(\{F_n\},\e)}{|\log \e|}\\
  &=\limsup_{\e\to 0}\frac{\sup_{\mu\in M(\X,G)}\underline{R}_{\mu}(\{F_n\},\e)}{|\log \e|}
\end{align*}
and
\begin{align*}
  {\rm \underline{mdim}_M}(\X,G,d)&\ge\liminf_{\e\to 0}\frac{\sup_{\mu\in M(\X,G)}\overline{R}_{\mu}(\{F_n\},\e)}{|\log \e|}\\
  &=\liminf_{\e\to 0}\frac{\sup_{\mu\in M(\X,G)}\overline{R}_{\mu}(\{F_n\},\e)}{|\log \e|}.
\end{align*}

\end{proposition}

\subsection{The upper bound} \

\begin{proposition}\label{prop-upper}
  For any $\e>0,D>2$ and any F{\o}lner sequence $\{F_n\}$, there exists $\mu\in M(\X,G)$ such that
  \begin{align}\label{mu}
    \underline{R}_{\mu}(\{F_n\},\e/2)\ge (1-\frac{1}{D})\overline S\big(\X,\{F_n\},d,(32D+8)\e\big).
  \end{align}
\end{proposition}
\begin{proof}
Let $\{F_n\}$ be the F{\o}lner sequence constructed as in Lemma \ref{lemma-folner}.

For each $F_n$ we choose $S_n$ to be a maximal $(8D+2)\e$-separated set of $\X$ with respect to the metric $\bar d_{F_n}$. Then
\begin{align}\label{ineq-4-Sn}
  |S_n|\ge \#\big (\X,\bar d_{F_n},(16D+4)\e \big ).
\end{align}

Define $$\nu_n=\frac{1}{|S_n|}\sum_{x\in S_n}\delta_x$$
and
$$\mu_n=\frac{1}{|F_n|}\sum_{g\in F_n}\nu_n\circ g^{-1}.$$

Choose a convergence subsequence $\{\mu_{n_i}\}$ in the weak$^*$ topology and assume it converges to $\mu$. Then $\mu\in M(\X,G)$.
Meanwhile, we can make $\{F_{n_i}\}$ be a tempered subsequence of $\{F_n\}$. We will show that $\mu$ satisfies the inequality \eqref{mu}.

Let $\mathcal P=\{P_1,\ldots,P_M\}$ be a measurable partition of $\X$ with ${\rm diam}(P_m,d)<\e$ and $\mu(\partial P_m)=0$ for each $1\le m\le M$.

Assign each $P_m$ a point $p_m\in P_m$ and set $A=\{p_1,\ldots,p_M\}$. Denote by $\mathcal P(x)=p_m$ for $x\in P_m$. Then
\begin{align}\label{ineq-4-Px}
  d\big (x,\mathcal P(x)\big )<\e, \text{ for any }x\in \X.
\end{align}
Let $\mathcal P^{F}(x)=\big(\mathcal P(gx)\big)_{g\in F}$ for $F\in F(G)$. Recall that we also use $\bar d_F$ to denote the metric on $\X^F$ for $F\in F(G)$ (see \eqref{metric} for the definition). By \eqref{ineq-4-Px}, we have $\bar d_{F_n}\big((gx)_{g\in F_n}, \mathcal P^{F_n}(x)\big)<\e$ for any $x\in\X$.
For any two distinct points $x,y\in S_n$, we have
\begin{align*}
  \bar d_{F_n}\big(\mathcal P^{F_n}(x),\mathcal P^{F_n}(y)\big)
  &\ge \bar d_{F_n}(x,y)-\bar d_{F_n}\big((gx)_{g\in F_n}, \mathcal P^{F_n}(x)\big)-\bar d_{F_n}\big((gy)_{g\in F_n}, \mathcal P^{F_n}(y)\big)\\
  &>(8D+2)\e-2\e=8D\e.
\end{align*}
Hence the set
$$\mathcal P^{F_n}(S_n)=\{\mathcal P^{F_n}(x)| x\in S_n\}$$
is an $8D\e$-separated set of $\X^{F_n}$ with respect to the metric $\bar d_{F_n}$.
Moreover, since $\nu_{F_n}$ is the uniform distribution over $S_n$, the push-forward measure $\mathcal P^{F_n}_* \nu_n$ is also the uniform distribution measure over
$\mathcal P^{F_n}(S_n)$. Note that $|\mathcal P^{F_n}(S_n)|=|S_n|$.

Let $X: \Omega\rightarrow \X$ be a random variable defined on some probability space $(\Omega,\p)$ such that the law of $X$ is given by $\mu$.
For $F\in F(G)$, let $Y_{F,g}: \Omega\rightarrow \X$ ($g\in F$) be random variables defined on the same probability space $(\Omega,\p)$ such that
$Y_F:=(Y_{F,g})_{g\in F}$  and $X$ are $(F,\e)$-close, i.e.
\begin{align}\label{ineq-4-XY}
  \E\bigg(\frac{1}{|F|}\sum_{g\in F}d(gX,Y_{F,g})\bigg)<\e.
\end{align}
We can assume the distribution of $Y_F$ is supported on a finite set $\Y_F\subset \X^F$ (by \cite[Remark 2.3]{LT2}).
By (1) of Proposition \ref{prop-mutual1}, the Data-processing inequality,
$$I(X;Y_F)\ge I\big(\mathcal P^F(X);Y_F\big).$$
Let $\tau_F
$ be the law of $\big(\mathcal P^F(X),Y_F\big)$, which is a probability measure on $A^F\times \Y_F$.
It follows that
\begin{align}\label{ineq-4-tau}
  \int_{A^F\times \Y_F}\bar d_F(x,y)d\tau_F(x,y)&=\E\bigg(\frac{1}{|F|}\sum_{g\in F}d\big(\mathcal P(gX),Y_{F,g}\big)\bigg)\nonumber\\
  &\le \E\bigg(\frac{1}{|F|}\sum_{g\in F}d(\mathcal P(gX),gX)\bigg)+\E\bigg(\frac{1}{|F|}\sum_{g\in F}d\big(gX,Y_{F,g}\big)\bigg)\nonumber\\
  &<2\e\text{ (by \eqref{ineq-4-Px} and \eqref{ineq-4-XY}) }.
\end{align}

For each $n\ge 1$, we consider the couplings of $(\mathcal P^F_*\mu_n,\mathcal P^F_*\mu)$ (i.e. a probability measure on $A^F\times A^F$ whose marginals are
  $\mathcal P^F_*\mu_n$ and $\mathcal P^F_*\mu$ respectively). We choose a probability measure $\pi_{F,n}$ that minimizes the following integral
  $$\int_{A^F\times A^F}\bar d_F(x,y)d\pi(x,y)$$
  among all such couplings $\pi$. Similar to Claim 30 of \cite{LT2}, the sequence $\pi_{F,n_i}$ converges to $(\mathcal P^F\times\mathcal P^F)_*\mu$ in the weak$^*$ topology.

Since both the second marginal of $\pi_{F,n}$ and the first marginal of $\tau_F$ are equal to the measure $\mathcal P^F_*\mu$,
we can compose them to produce a coupling $\tau_{F,n}$ of $\big(\mathcal P^F_*\mu_n,{\rm Law}(Y)\big)$ by the following way:
$$\tau_{F,n}(x,y)=\sum_{x'\in A^F}\pi_{F,n}(x,x')\p(Y=y|\mathcal P^F(X)=x'), \qquad (x\in A^F,y\in\Y_F).$$
We note here that the sequence $\tau_{F,n_i}$ converges to $\tau_F$ in the weak$^*$ topology and hence by \eqref{ineq-4-tau},
\begin{align}\label{ineq-4-tau2}
  \E_{\tau_{F,n_i}}\big(\bar d_F(x,y)\big):=\int_{A^F\times \Y_F}\bar d_F(x,y)d\tau_{F,n_i}(x,y)<2\e
\end{align}
for all sufficiently large $n_i$.

For $x\in \bigcup_{g\in F_n}\mathcal P^F(gS_n)$ and $y\in \X^F$, define a conditional probability mass function $\tau_{F,n}(y|x)$ by
$$\tau_{F,n}(y|x)=\frac{\tau_{F,n}(x,y)}{\mathcal P^F_*\mu_n(x)}.$$

Recall that our F{\o}lner sequence $\{F_n\}$ is constructed as in Lemma \ref{lemma-folner}.
Then for any $K\in F(G)$ with $e_G\in K$ and $0<\e_1<\min\{\frac{1}{2},\frac{\e}{{\rm diam}(\X,d)}\}$,
by Lemma \ref{lemma-folner} (here we choose $\{H_n\}$ to be the tempered F{\o}lner sequence $\{F_{n_i}\}$), there exists $\mathcal T$, a finite tiling of $G$,
satisfying the following two conditions:
 \begin{enumerate}
  \item[(C1)] $\mathcal{T}$ has shapes $\{F_{m_1},\ldots,F_{m_l}\}$ consisted with F{\o}lner sets in $\{F_n\}$ and each $F_{m_j}$
  is $(K,\e_1)$-invariant;
  \item[(C2)] for sufficiently large $i$, for each $1\le j\le l$, the family of sets $\{C_jg^{-1}\}_{g\in F_{n_i}}$ covers a subset $\tilde{F}_{n_i}\subset F_{n_i}$ with
  $|\tilde{F}_{n_i}|>(1-\e_1)|F_{n_i}|$ no more than $(1+\e_1)\rho_{\mathcal T}(F_{m_j},F_{n_i})\frac{|F_{n_i}|}{|F_{m_j}|}$-many times, where $C_j$ is the center of the shape $F_{m_j}$.
\end{enumerate}
Note that $\mathcal T=\{F_{m_j}c: c\in C_j,1\le j\le l\}$ and
$$G=\bigsqcup_{j=1}^l\bigsqcup_{c\in C_j}F_{m_j}c.$$
Here `` $\bigsqcup$ " stands for the disjoint union.
For $g\in F_{n_i}$, denote by $$R_g=F_{n_i}\setminus \bigg(\bigsqcup_{j=1}^l\bigsqcup_{c\in C_j,F_{m_j}c\subset F_{n_i}g,cg^{-1}\in \tilde{F}_{n_i}}F_{m_j}cg^{-1}\bigg),$$
i.e. $R_g$ is the remaining part by removing from $F_{n_i}$ the elements of $\mathcal Tg^{-1}$ that entirely contained in $F_{n_i}$.
Obviously, $R_g\subset B(F_{n_i}, \cup_{j=1}^lF_{m_j})$, and thus when $n_i$ is large enough, $|R_g|< \e |F_{n_i}|$.

Fix a point $a\in \X$. For $x=(x_g)_{g\in F_{n_i}}\in \mathcal P^{F_{n_i}}(S_{n_i})$ and $g\in F_{n_i}$, we define probability mass functions $\sigma_{F_{n_i},g}(\cdot|x)$ on $\X^{F_{n_i}}$ as the following:
for $y=(y_g)_{g\in F_{n_i}}\in \X^{F_{n_i}}$,
\begin{align}\label{def-measure1}
  \sigma_{F_{n_i},g}(y|x)=\prod_{j=1}^l\prod_{c\in C_j,F_{m_j}c\subset F_{n_i}g,cg^{-1}\in \tilde{F}_{n_i}}\tau_{F_{m_j},n_i}(y_{F_{m_j}cg^{-1}}|x_{F_{m_j}cg^{-1}})\cdot \prod_{h\in R_g}\delta_a(y_h).
\end{align}
Here we note that $$y_{F_{m_j}cg^{-1}}=(y_h)_{h\in F_{m_j}cg^{-1}}\in \X^{F_{m_j}cg^{-1}}$$ and
$$x_{F_{m_j}cg^{-1}}=(x_h)_{h\in F_{m_j}cg^{-1}}\in \mathcal P^{F_{m_j}cg^{-1}}(S_{n_i})\subset \X^{F_{m_j}cg^{-1}}.$$
Then we set
\begin{align}\label{def-measure2}
  \sigma_{F_{n_i}}(y|x)=\frac{1}{|F_{n_i}|}\sum_{g\in F_{n_i}}\sigma_{F_{n_i},g}(y|x).
\end{align}

\begin{claim}\label{claim-I}
  For sufficiently large $n_i$, there exists some $1\le j\le l$ such that
\begin{align*}
  (1-\e_1)\frac{1}{|F_{n_i}|}I(\mathcal P^{F_{n_i}}_*\nu_n,\sigma_{F_{n_i}})\le \frac{1}{|F_{m_j}|}I\big(\mathcal P^{F_{m_j}}_*(\mu_{n_i}),\tau_{F_{m_j},n_i}\big).
\end{align*}
\end{claim}
\begin{proof}[{\bf Proof of Claim \ref{claim-I}}]
By (2) of Proposition \ref{prop-mutual2}, the convexity of mutual information,
\begin{align}\label{ineq-4-I1}
  I(\mathcal P^{F_{n_i}}_*\nu_{n_i},\sigma_{F_{n_i}})\le \frac{1}{|F_{n_i}|}\sum_{g\in F_{n_i}}I(\mathcal P^{F_{n_i}}_*\nu_{n_i},\sigma_{F_{n_i},g}).
\end{align}
By (5) of Proposition \ref{prop-mutual1}, the subadditivity of mutual information, together with \eqref{def-measure1}, we have
\begin{align}\label{ineq-4-I2}
  I(\mathcal P^{F_{n_i}}_*\nu_{n_i},\sigma_{F_{n_i},g})\le \sum_{j=1}^l\sum_{c\in C_j,F_{m_j}c\subset F_{n_i}g, cg^{-1}\in\tilde F_{n_i}}I\bigg(\mathcal P^{F_{m_j}}_*\big((cg^{-1})_*\nu_{n_i}\big),\tau_{F_{m_j},n_i}\bigg).
\end{align}

Joint \eqref{ineq-4-I1} and \eqref{ineq-4-I2} together,
\begin{align*}
  I(\mathcal P^{F_{n_i}}_*\nu_{n_i},\sigma_{F_{n_i}})&\le \frac{1}{|F_{n_i}|}\sum_{g\in F_{n_i}}\sum_{j=1}^l\sum_{c\in C_j,F_{m_j}c\subset F_{n_i}g,cg^{-1}\in\tilde F_{n_i}}
  I\bigg(\mathcal P^{F_{m_j}}_*\big((cg^{-1})_*\nu_{n_i}\big),\tau_{F_{m_j},n_i}\bigg)\\
  &=\frac{1}{|F_{n_i}|}\sum_{j=1}^l\sum_{g\in F_{n_i}}\sum_{c\in C_j,F_{m_j}c\subset F_{n_i}g,cg^{-1}\in\tilde F_{n_i}}
  I\bigg(\mathcal P^{F_{m_j}}_*\big((cg^{-1})_*\nu_{n_i}\big),\tau_{F_{m_j},n_i}\bigg).
\end{align*}
For convenience, denote by $t_j=\rho_{\mathcal T}(F_{m_j},F_{n_i})$, then we have
\begin{align*}
  I(\mathcal P^{F_{n_i}}_*\nu_n,\sigma_{F_{n_i}})
  &\le \frac{1}{|F_{n_i}|}\sum_{j=1}^l\sum_{h\in \tilde{F}_{n_i}}(1+\e_1)t_j\frac{|F_{n_i}|}{|F_{m_j}|}
  I\big(\mathcal P^{F_{m_j}}_*(h_*\nu_{n_i}),\tau_{F_{m_j},n_i}\big)\\
  &\qquad \text{ (by condition (C2))}\\
  &\le \sum_{j=1}^l (1+\e_1)t_j\frac{|F_{n_i}|}{|F_{m_j}|} \frac{1}{|F_{n_i}|}\sum_{h\in F_{n_i}}
  I\big(\mathcal P^{F_{m_j}}_*(h_*\nu_{n_i}),\tau_{F_{m_j},n_i}\big)\\
  &\le \sum_{j=1}^l(1+\e_1)t_j\frac{|F_{n_i}|}{|F_{m_j}|}I\big(\mathcal P^{F_{m_j}}_*(\frac{1}{|F_{n_i}|}\sum_{h\in F_{n_i}}h_*\nu_{n_i}),\tau_{F_{m_j},n_i}\big)\\
  &\qquad \text{(by (1) of Proposition \ref{prop-mutual2}, the concavity of mutual information)}\\
  &=(1+\e_1)|F_{n_i}|\sum_{j=1}^lt_j\frac{1}{|F_{m_j}|}I\big(\mathcal P^{F_{m_j}}_*(\mu_{n_i}),\tau_{F_{m_j},n_i}\big),
\end{align*}
i.e.
\begin{align*}
  (1-\e_1)\frac{1}{|F_{n_i}|}I(\mathcal P^{F_{n_i}}_*\nu_n,\sigma_{F_{n_i}})\le\sum_{j=1}^lt_j\frac{1}{|F_{m_j}|}I\big(\mathcal P^{F_{m_j}}_*(\mu_{n_i}),\tau_{F_{m_j},n_i}\big).
\end{align*}

Noticing that $\sum_{j=1}^lt_j\le 1$, there must exists some $1\le j\le l$ such that
\begin{align*}
  (1-\e_1)\frac{1}{|F_{n_i}|}I(\mathcal P^{F_{n_i}}_*\nu_n,\sigma_{F_{n_i}})\le \frac{1}{|F_{m_j}|}I(\mathcal P^{F_{m_j}}_*(\mu_n),\tau_{F_{m_j},n}).
\end{align*}
This finishes the proof of Claim \ref{claim-I}.
\end{proof}
Denote by $\E_{\mathcal P^{F_{n_i}}_*\nu_{n_i},\sigma_{F_{n_i}}}\big(\bar d_{F_{n_i}}(x,y)\big)$ the expected value of $\bar d_{F_{n_i}}(x,y)$ ($x,y\in \X^{F_{n_i}}$) with respect to the probability measure
$\mathcal P^{F_{n_i}}_*\nu_{n_i}(x)\sigma_{F_{n_i}}(y|x)$.

\begin{claim}\label{claim-E}
For sufficiently large $n_i$,
\begin{align*}
  \E_{\mathcal P^{F_{n_i}}_*\nu_{n_i},\sigma_{F_{n_i}}}\big(\bar d_{F_{n_i}}(x,y)\big)< 4\e.
\end{align*}
and
\begin{align*}
  I(\mathcal P^{F_{n_i}}_*\nu_{n_i},\sigma_{F_{n_i}})\ge (1-\frac{1}{D})\log |S_{n_i}|-H(\frac{1}{D}).
\end{align*}
\end{claim}
\begin{proof}[{\bf Proof of Claim \ref{claim-E}}]
  By \eqref{def-measure1} and \eqref{def-measure2}, the definitions of probability mass functions $\sigma_{F_{n_i},g}(\cdot|x)$ ($g\in F_{n_i}$) and $\sigma_{F_{n_i}}(\cdot|x)$, we have
\begin{align*}
  \E_{\mathcal P^{F_{n_i}}_*\nu_{n_i},\sigma_{F_{n_i}}}\big(\bar d_{F_{n_i}}(x,y)\big)
  =\frac{1}{|F_{n_i}|}\sum_{g\in F_{n_i}}\E_{\mathcal P^{F_{n_i}}_*\nu_{n_i},\sigma_{F_{n_i},g}}\big(\bar d_{F_{n_i}}(x,y)\big)
\end{align*}
and
\begin{align*}
  &|F_{n_i}|\E_{\mathcal P^{F_{n_i}}_*\nu_{n_i},\sigma_{F_{n_i},g}}\big(\bar d_{F_{n_i}}(x,y)\big)\\
  \le & \sum_{j=1}^l\sum_{c\in C_j,F_{m_j}c\subset F_{n_i}g,cg^{-1}\in\tilde F_{n_i}}|F_{m_j}|\E_{\mathcal P^{F_{m_j}}_*\big((cg^{-1})_*\nu_{n_i}\big),\tau_{F_{m_j},n_i}}
  \big(\bar d_{F_{m_j}}(x',y')\big)\\
  &\qquad +|R_g|{\rm diam}(\X,d),
\end{align*}
where $x,y$ are random points in $\X^{F_{n_i}}$ and $x',y'$ appeared in $\bar d_{F_{m_j}}(x',y')$ are in $\X^{F_{m_j}}$.

When $F_{n_i}$ is sufficiently invariant, $|R_g|<\e_1|F_{n_i}|$. Hence
\begin{align*}
  &\E_{\mathcal P^{F_{n_i}}_*\nu_{n_i},\sigma_{F_{n_i}}}\big(\bar d_{F_{n_i}}(x,y)\big)\\
  \le& \frac{1}{|F_{n_i}|}\sum_{j=1}^l\sum_{g\in F_{n_i}}\sum_{c\in C_j,F_{m_j}c\subset F_{n_i}g,cg^{-1}\in\tilde F_{n_i}}
  \frac{|F_{m_j}|}{|F_{n_i}|}\E_{\mathcal P^{F_{m_j}}_*((cg^{-1})_*\nu_{n_i}),\tau_{F_{m_j},n_i}}\big(\bar d_{F_{m_j}}(x',y')\big)\\&\qquad +\e_1{\rm diam}(\X,d)\\
  \le& \frac{1}{|F_{n_i}|}\sum_{j=1}^l\sum_{h\in \tilde F_{n_i}}(1+\e_1)t_j
  \E_{\mathcal P^{F_{m_j}}_*(h_*\nu_{n_i}),\tau_{F_{m_j},n_i}}\big(\bar d_{F_{m_j}}(x',y')\big)+\e_1{\rm diam}(\X,d)\\
  &\qquad \text{ (by condition (C2) and recall here } t_j=\rho_{\mathcal T}(F_{m_j},F_{n_i})\text{)}\\
  \le& \sum_{j=1}^l(1+\e_1)t_j\E_{\mathcal P^{F_{m_j}}_*(\frac{1}{|F_{n_i}|}\sum_{h\in F_{n_i}}h_*\nu_{n_i}),\tau_{F_{m_j},n_i}}\big(\bar d_{F_{m_j}}(x',y')\big)+\e_1{\rm diam}(\X,d)\\
  =&\sum_{j=1}^l(1+\e_1)t_j\E_{\mathcal P^{F_{m_j}}_*\mu_{n_i},\tau_{F_{m_j},n_i}}\big(\bar d_{F_{m_j}}(x',y')\big)+\e_1{\rm diam}(\X,d)\\
  =&\sum_{j=1}^l(1+\e_1)t_j\int_{A^{F_{m_j}}\times \Y_{F_{m_j}}}\bar d_{F_{m_j}}(x',y')d\tau_{F_{m_j},n_i}(x,y)+\e_1{\rm diam}(\X,d).
\end{align*}

Recall that $0<\e_1<\min\{\frac{1}{2},\frac{\e}{{\rm diam}(\X,d)}\}$ and $\sum_{j=1}^lt_j\le 1$. By \eqref{ineq-4-tau2}, for sufficiently large $n_i$, we have
\begin{align*}
  \E_{\mathcal P^{F_{n_i}}_*\nu_{n_i},\sigma_{F_{n_i}}}\big(\bar d_{F_{n_i}}(x,y)\big)< (1+\frac{1}{2})2\e+\e=4\e.
\end{align*}

Since $\mathcal P^{F_{n_i}}_*\nu_{n_i}$ is uniformly distributed over $\mathcal P^{F_{n_i}}(S_{n_i})$ and $\mathcal P^{F_{n_i}}(S_{n_i})$ is a $(8D\e)$-separated set of cardinality $|S_{n_i}|$,
by Lemma \ref{lemma-Fano}, for sufficiently large $n_i$,
\begin{align*}
  I(\mathcal P^{F_{n_i}}_*\nu_{n_i},\sigma_{F_{n_i}})\ge (1-\frac{1}{D})\log |S_{n_i}|-H(\frac{1}{D}).
\end{align*}
 This finishes the proof of Claim \ref{claim-E}.
\end{proof}

Now we proceed with the proof of Lemma \ref{prop-upper}.

For any $K\in F(G)$ with $e_G\in K$ and $0<\e_1<\min\{\frac{1}{2},\frac{\e}{{\rm diam}(\X,d)}\}$, for sufficiently large $n_i$, there exists a $1\le j\le l$
(here $j$ depends on $n_i$, whereas $l$ depends on $K$ and $\e_1$ but does not depend on $n_i$),
\begin{align*}
  &\frac{1}{|F_{m_j}|}I(\mathcal P^{F_{m_j}}_*(\mu_{n_i}),\tau_{F_{m_j},n_i})\\
  \ge & (1-\e_1)\frac{1}{|F_{n_i}|}I(\mathcal P^{F_{n_i}}_*\nu_{n_i},\sigma_{F_{n_i}})\text{ (by Claim \ref{claim-I})}\\
  \ge & (1-\e_1)\bigg((1-\frac{1}{D})\frac{\log |S_{n_i}|}{|F_{n_i}|}-\frac{H(\frac{1}{D})}{|F_{n_i}|}\bigg) \text{ (by Claim \ref{claim-E})}\\
  \ge & (1-\e_1)\bigg( (1-\frac{1}{D})\frac{\log \#\big(\X,\bar d_{F_{n_i}},(16D+4)\e\big)}{|F_{n_i}|} - \frac{H(\frac{1}{D})}{|F_{n_i}|}\bigg) \text{ (by \eqref{ineq-4-Sn})}.
\end{align*}

By choosing some subsequence of $\{n_i\}$ (we still denote it by $\{n_i\}$), for some $1\le j\le l$, the probability measures $\tau_{F_{m_j},n_i}$
converge to  $\tau_{F_{m_j}}={\rm Law}\big(\mathcal P^{F_{m_j}}(X),Y_{F_{m_j}}\big)$ in the weak$^*$ topology. Let $n_i\rightarrow \infty$. By (3) of Proposition \ref{prop-mutual1},
\begin{align*}
  \frac{1}{|F_{m_j}|}I\big(\mathcal P^{F_{m_j}}(X);Y_{F_{m_j}}\big)\ge (1-\e_1)(1-\frac{1}{D})\overline S \big( \X,\{F_{n_i}\},d,(16D+4)\e \big).
\end{align*}
By (1) of Proposition \ref{prop-mutual1}, the data-processing inequality,
\begin{align*}
  \frac{1}{|F_{m_j}|}I(X;Y_{F_{m_j}})\ge (1-\e_1)(1-\frac{1}{D})\overline S \big( \X,\{F_{n_i}\},d,(16D+4)\e \big).
\end{align*}
Let $(K,\e_1)$ be chosen from the pairs $(K_n,\frac{1}{n})$, where $K_n\in F(G)$ and $K_n\uparrow G$. The $F_{m_j}$'s above subject to $(K_n,\frac{1}{n})$ form a new F{\o}lner sequence. We denote this new F{\o}lner sequence by $\{T_n\}$ and then it follows that
\begin{align*}
  \underline R_{\mu}(\{T_n\},\e)\ge (1-\frac{1}{D})\overline S\big(\X,\{F_{n_i}\},d,(16D+4)\e\big).
\end{align*}
Passing $\{T_n\}$ and $\{F_{n_i}\}$ to any F{\o}lner sequence by Proposition \ref{prop-S} and Proposition \ref{prop-R},
we complete the proof of Proposition \ref{prop-upper}.
\end{proof}

\begin{proof}[\bf {Proof of Theorem \ref{theorem-L1}}] \

For $D>2$ and any F{\o}lner sequence $\{F_n\}$,
\begin{align*}
  \limsup_{\e\to 0}\frac{\sup_{\mu\in M(\X,G)}\underline R_{\mu}(\{F_n\},\e)}{|\log \frac{\e}{2}|}
  &\ge \limsup_{\e\to 0}\frac{(1-\frac{1}{D})\overline S\big(\X,\{F_n\},d,(32D+8)\e\big)}{|\log \e|}\\
  &\qquad \text{ (by Proposition \ref{prop-upper})}\\
  &=(1-\frac{1}{D})\limsup_{\e\to 0}\frac{\overline S\big(\X,\{F_n\},d,(32D+8)\e\big)}{|\log (32D+8)\e|}\\
  &=(1-\frac{1}{D}){\rm \overline{mdim}_M}(\X,G,d)\text{ (by Proposition \ref{mdim-equal})}.
\end{align*}
Leting $D\rightarrow \infty$, we have
\begin{align*}
  {\rm \overline{mdim}_M}(\X,G,d)\le\limsup_{\e\to 0}\frac{\sup_{\mu\in M(\X,G)}\underline R_{\mu}(\{F_n\},\e)}{|\log \e|}.
\end{align*}
And similarly,
\begin{align*}
  {\rm \underline{mdim}_M}(\X,G,d)\le\liminf_{\e\to 0}\frac{\sup_{\mu\in M(\X,G)}\underline R_{\mu}(\{F_n\},\e)}{|\log \e|}.
\end{align*}
Joint with Proposition \ref{prop-lower}, we obtain
\begin{align*}
  {\rm \overline{mdim}_M}(\X,G,d)&=\limsup_{\e\to 0}\frac{\sup_{\mu\in M(\X,G)}\overline{R}_{\mu}(\{F_n\},\e)}{|\log \e|}\\
  &=\limsup_{\e\to 0}\frac{\sup_{\mu\in M(\X,G)}\underline{R}_{\mu}(\{F_n\},\e)}{|\log \e|}
\end{align*}
and
\begin{align*}
  {\rm \underline{mdim}_M}(\X,G,d)&=\liminf_{\e\to 0}\frac{\sup_{\mu\in M(\X,G)}\overline{R}_{\mu}(\{F_n\},\e)}{|\log \e|}\\
  &=\liminf_{\e\to 0}\frac{\sup_{\mu\in M(\X,G)}\underline{R}_{\mu}(\{F_n\},\e)}{|\log \e|}.
\end{align*}

\end{proof}

\section{$L^{\infty}$ and $L^p$ $(p\ge 1)$ variational principles}

Modifying the distortion condition \eqref{condition-rate}, we can also define $L^{\infty}$ and $L^p$ $(p\ge 1)$ rate distortion functions.
Similarly, we have $L^{\infty}$ and $L^p$ $(p\ge 1)$ variational principles between metric mean dimensions and the corresponding rate distortion functions.

Let $(\mathcal{X},G)$ be a $G$-system with metric $d$. We define the {\it $L^{\infty}$ rate distortion function} of $(\mathcal{X},G)$ in the following way.

Let $\e>0$ and $\mu\in M(\X,G)$. For $F\in F(G)$, let $X: \Omega\rightarrow \X$ and $Y_g: \Omega\rightarrow \X, g\in F$ be random variables defined on some probability space $(\Omega,\p)$.
Assume $\mu={\rm Law} (X)$. We say $X$ and $Y=(Y_g)_{g\in F}$ are $(F,\e,\alpha)_{L^{\infty}}$-close for $\alpha>0$ if
  \begin{align*}
    \E\bigg(\frac{1}{|F|}\#\{g\in F:d(gX,Y_g)\ge\e\}\bigg)<\alpha.
  \end{align*}
Denote by $$R_{\mu,\infty}(\e,\alpha,F)=\inf_{X,Y \text{ are }(F,\e,\alpha)_{L^{\infty}}\text{-close}}I(X;Y).$$
For a F{\o}lner sequence $\{F_n\}$, we define
\begin{align*}
  \underline{R}_{\mu,\infty}(\{F_n\},\e,\alpha)=\liminf_{n\rightarrow\infty}\frac{R_{\mu,\infty}(\e,\alpha,F_n)}{|F_n|}
\end{align*}
and
\begin{align*}
  \overline{R}_{\mu,\infty}(\{F_n\},\e,\alpha)=\limsup_{n\rightarrow\infty}\frac{R_{\mu,\infty}(\e,\alpha,F_n)}{|F_n|}.
\end{align*}

\begin{proposition}\label{prop-R-infinity}
  Let $\{F_n\}$ and $\{H_n\}$ be any two F{\o}lner sequences and $\e,\alpha>0$. Then
  $$\overline{R}_{\mu,\infty}(\{H_n\},\e,2\alpha)\le \underline{R}_{\mu,\infty}(\{F_n\},\e,\alpha).$$
\end{proposition}
The proof is similar to that of Proposition \ref{prop-R}, we omit it here.

Since both $\underline{R}_{\mu,\infty}(\{F_n\},\e,\alpha)$ and $\overline{R}_{\mu,\infty}(\{F_n\},\e,\alpha)$ do not increase as $\alpha$ decreases, 
by Proposition \ref{prop-R-infinity}, the following limit exist and do not depend on the choice of the F{\o}lner sequence $\{F_n\}$.
$$R_{\mu,\infty}(\e):=\lim_{\alpha\rightarrow 0}\underline{R}_{\mu,\infty}(\{F_n\},\e,\alpha)=\lim_{\alpha\rightarrow 0}\overline{R}_{\mu,\infty}(\{F_n\},\e,\alpha).$$
We call $R_{\mu,\infty}(\e)$ the {\it $L^{\infty}$ rate distortion function}. 

The following theorem is the $L^{\infty}$ variational principles for metric mean dimension. The proof uses the same spirit of Theorem \ref{theorem-L1}). Since the $\bar d_F$ metric and $\tilde S(\X,G,d,\e)$ are not involved, 
the proof is simpler than that of Theorem \ref{theorem-L1} (but it is still complicated). We will put the proof in Appendix A. We note that for this theorem,  $(\X,d)$ need not have tame growth of covering numbers.
\begin{theorem}\label{theorem-Linfty}
\begin{align*}
  {\rm \overline{mdim}_M}(\X,G,d)&=\limsup_{\e\to 0}\frac{\sup_{\mu\in M(\X,G)}R_{\mu,\infty}(\e)}{|\log \e|}
\end{align*}
and
\begin{align*}
  {\rm \underline{mdim}_M}(\X,G,d)&=\liminf_{\e\to 0}\frac{\sup_{\mu\in M(\X,G)}R_{\mu,\infty}(\e)}{|\log \e|}.
\end{align*}
\end{theorem}
\begin{proof}
  See Appendix A.
\end{proof}

Fix $1\le p<\infty$. Let $F\in F(G)$, $X: \Omega\rightarrow \X$ and $Y_g: \Omega\rightarrow \X, g\in F$ be given as previous.
We say $X$ and $Y=(Y_g)_{g\in F}$ are $(F,\e)_{L^p}$-close if
  \begin{align*}
    \E\bigg(\frac{1}{|F|}\sum_{g\in F}d(gX,Y_g)^p\bigg)<\e^p.
  \end{align*}

Denote by $$R_{\mu,p}(\e,F)=\inf_{X,Y \text{ are }(F,\e)_{L^p}\text{-close}}I(X;Y).$$
The {\it $L^p$ rate distortion functions} are then defined
by $$\underline R_{\mu,p}(\{F_n\},\e)=\liminf_{n\rightarrow \infty}\frac{R_{\mu,p}(\e,F_n)}{|F_n|}
\text{ and } 
\overline R_{\mu,p}(\{F_n\}\e)=\limsup_{n\rightarrow \infty}\frac{R_{\mu,p}(\e,F_n)}{|F_n|},$$
where $\{F_n\}$ is a F{\o}lner sequence in $G$. 
When $p=1$, $\overline R_{\mu,1}(\{F_n\},\e)$ and $\underline R_{\mu,1}(\{F_n\},\e)$ coincide with $\overline R_{\mu}(\{F_n\},\e)$ and $\underline R_{\mu}(\{F_n\},\e)$ defined in Section 3, respectively.

Similar to Proposition \ref{prop-R}, we have
\begin{proposition}\label{prop-R-Lp}
  Let $\{F_n\}$ and $\{H_n\}$ be any two F{\o}lner sequences, $p\ge 1$ and $\e>0$. Then
  $$\overline{R}_{\mu,p}(\{H_n\},2\e)\le \underline{R}_{\mu,p}(\{F_n\},\e).$$
\end{proposition}

Applying the $L^1$ and $L^{\infty}$ variational principles, we can obtain the following $L^p$ $(p\ge1)$ variational principles under the condition that
$(\X,d)$ has tame growth of covering numbers.

\begin{theorem}\label{theorem-Lp}
  If $(\X,d)$ has tame growth of covering numbers, then for any $p\ge 1$,
  \begin{align*}
  {\rm \overline{mdim}_M}(\X,G,d)&=\limsup_{\e\to 0}\frac{\sup_{\mu\in M(\X,G)}\overline{R}_{\mu,p}(\{F_n\},\e)}{|\log \e|}\\
&=\limsup_{\e\to 0}\frac{\sup_{\mu\in M(\X,G)}\underline{R}_{\mu,p}(\{F_n\},\e)}{|\log \e|}
\end{align*}
and
\begin{align*}
  {\rm \underline{mdim}_M}(\X,G,d)&=\liminf_{\e\to 0}\frac{\sup_{\mu\in M(\X,G)}\overline{R}_{\mu,p}(\{F_n\},\e)}{|\log \e|}\\
&=\liminf_{\e\to 0}\frac{\sup_{\mu\in M(\X,G)}\underline{R}_{\mu,p}(\{F_n\},\e)}{|\log \e|},
\end{align*}
where $\{F_n\}$ is any F{\o}lner sequence.
\end{theorem}
\begin{proof}
  Let $p\ge 1,\alpha>0,\e>0$ and $\mu\in M(\X,G)$. For $F\in F(G)$, let $X: \Omega\rightarrow \X$ and $Y_g: \Omega\rightarrow \X, g\in F$ be random variables as in the definition of the rate distortion functions.

  If $X$ and $Y=(Y_g)_{g\in F}$ are $(F,\e)_{L^{p}}$-close,
  then by the H\"{o}lder inequality, it holds that
  \begin{align*}
    \E\bigg(\frac{1}{|F|}\sum_{g\in F}d(gX,Y_g)\bigg)<\bigg(\E\big(\frac{1}{|F|}\sum_{g\in F}d(gX,Y_g)^p\big)\bigg )^{\frac{1}{p}}<\e,
  \end{align*}
  i.e. $X$ and $Y=(Y_g)_{g\in F}$ are $(F,\e)$-close. And hence by the definition of the rate distortion functions,
  $$\underline R_{\mu}(\{F_n\},\e)\le \underline R_{\mu,p}(\{F_n\},\e),$$
  for any F{\o}lner sequence $\{F_n\}$.

  If $X$ and $Y=(Y_g)_{g\in F}$ are $(F,\e,\alpha)_{L^{\infty}}$-close for $\alpha>0$, i.e.
  \begin{align*}
    \E\bigg(\frac{1}{|F|}\#\{g\in F:d(gX,Y_g)\ge\e\}\bigg)<\alpha,
  \end{align*}
  then
  \begin{align*}
    \frac{1}{|F|}\sum_{g\in F}d(gX,Y_g)^p&\le \e^p+\frac{1}{|F|}\sum_{g\in F, d(gX,Y_g)\ge\e}d(gX,Y_g)^p\\
    &\le \e^p+\frac{1}{|F|}\#\{g\in F:d(gX,Y_g)\ge\e\}\cdot \big({\rm diam}(\X,d)\big)^p.
  \end{align*}
  And hence
  \begin{align*}
    \E\bigg(\frac{1}{|F|}\sum_{g\in F}d(gX,Y_g)^p\bigg)< \e^p+\alpha\big({\rm diam}(\X,d)\big)^p.
  \end{align*}
  Then it follows that for any $\e'>\e$, when $\alpha$ is sufficiently small,
  \begin{align*}
\bigg(\E\big(\frac{1}{|F|}\sum_{g\in F}d(gX,Y_g)^p\big)\bigg)^{\frac{1}{p}}< \e'.
  \end{align*}
  Hence $$\underline R_{\mu,p}(\{F_n\},\e')\le R_{\mu,\infty}(\e), \text{ for any }\e'>\e.$$

  The conclusion then follows by Theorem \ref{theorem-L1}, Theorem \ref{theorem-Linfty} and Proposition \ref{prop-R-Lp}.
\end{proof}

{\bf Acknowledgements}
This research is supported by NNSF of China (Grant No. 11790274, 11701275), National Basic Research Program of China (Grant No. 2013CB
834100) and Tianyuan Mathematical Center in Southwest China. The authors would like to thank Prof. Wen Huang and Dr. Yunping Wang for their valuable discussions and comments.
This work was started when the third named author stayed in the School of Mathematics and Statistics, the University of Sheffield. She was grateful for the kindly support there.

\appendix
\renewcommand{\appendixname}{Appendix~\Alph{section}}
 \section{Proof of Theorem \ref{theorem-Linfty}}
For the proof of Theorem \ref{theorem-Linfty}, we need he following Lemma (\cite[Lemma 17]{LT2}).
 \begin{lemma}\label{lemma-mutual-Linfty}
   Let $(\X,d)$ be a compact metric space with a finite subset $A$. Let $F\in F(G),\e>0$ and $0<\alpha\le \frac{1}{2}$.
   Suppose $S\subset A^F$ is a $2\e$-separated set with respect to the metric $d_F\big((x_g)_{g\in F},(y_g)_{g\in F}\big)$.
   Let $X=(X_g)_{g\in F}$ and $y=(Y_g)_{g\in F}$ be measurable maps from $\Omega$ to $\X^F$ such that $X$ is uniformly distributed over $S$ and
   \begin{align*}
    \E \big (\#\{g\in F: d(X_g,Y_g)\ge \e\}\big )< \alpha|F|.
   \end{align*}
   Then
   $$I(X;Y)\ge \log |S|-|F|H(\alpha)-\alpha|F|\log |A|.$$
 \end{lemma}

  \begin{lemma}\label{lemma-lower-Linfty}
  For $\e>0$ and $\mu\in M(\X,G)$, we have
  $$R_{\mu,\infty}(\e)\le S(\X,G,d,\e).$$
\end{lemma}
\begin{proof}
Let $\{F_n\}$ be a F{\o}lner sequence in $G$. For $n>0$, denote by $M=\#(\X,d_{F_n},\e)$ and
let $\{U_1,\ldots,U_M\}$ be an open cover of $\X$ with ${\rm diam}(U_m,d_{F_n})<\e$ for each $1\le m\le M$.
Choose a point $p_m\in U_m$ for each $m$. For any $x\in \X$, let $m$ be the smallest number satisfying $x\in U_m$. Then by setting $f(x)=p_m$ we can define a map $f:\X\rightarrow \{p_1,\ldots,p_M\}$ and hence $d_{F_n}\big(x,f(x)\big)<\e$.
Let $X$ be a random variable with ${\rm Law}(X)=\mu$. Then $d_{F_n}\big(X,f(X)\big)<\e$ almost surely, which implies
$$\E\bigg(\frac{1}{|F_n|}\#\{g\in F_n:d\big(gX,gf(X)\big)\ge\e\}\bigg)=0.$$
Let $Y=\big(gf(X)\big)_{g\in F_n}$. Obviously $X$ and $Y$ are $(F_n,\e,\alpha)_{L^{\infty}}$-close for any $\alpha>0$.
Hence $$\overline{R}_{\mu,\infty}(\e,\alpha,F_n)\le I(X;Y)\le H(Y)\le \log M=\log\#(\X,d_{F_n},\e).$$
Dividing by $|F_n|$ and letting $n\rightarrow\infty$, we have
$$R_{\mu,\infty}(\e)\le S(\X,G,d,\e).$$
\end{proof}

\begin{proposition}\label{prop-upper-Linfty}
  For any $\e>0$ there exists $\mu\in M(\X,G)$ such that
  \begin{align}\label{mu-Linfty}
    R_{\mu,\infty}(\e)\ge S(\X,G,d,12\e).
  \end{align}
\end{proposition}
\begin{proof}
Let $\{F_n\}$ be the F{\o}lner sequence in $G$ constructed as in Lemma \ref{lemma-folner}.

For each $F_n$ we choose $S_n$ to be a maximal $6\e$-separated set of $\X$ with respect to the metric $d_{F_n}$. Then
\begin{align}\label{ineq-A-Sn-Linfty}
  |S_n|\ge \#(\X,d_{F_n},12\e).
\end{align}

Define $$\nu_n=\frac{1}{|S_n|}\sum_{x\in S_n}\delta_x$$
and
$$\mu_n=\frac{1}{|F_n|}\sum_{g\in F_n}\nu_n\circ g^{-1}.$$

As in the proof of Proposition \ref{prop-upper}, we first choose a tempered subsequence $\{F_{n_i}\}$ of $\{F_n\}$,
then choose a convergence subsequence of $\{\mu_{n_i}\}_{i=1}^{\infty}$ in the weak$^*$ topology and assume it converges to $\mu$.
Hence $\mu\in M(\X,G)$ and we will show it satisfies the inequality \eqref{mu-Linfty}.
For simplicity, we still denote this subsequence by $\{\mu_{n_i}\}_{i=1}^{\infty}$.

Let $\mathcal P=\{P_1,\ldots,P_M\}$ be a measurable partition of $\X$ with ${\rm diam}(P_m,d)<\e$ and $\mu(\partial P_m)=0$ for each $1\le m\le M$.

Assign each $P_m$ a point $p_m\in P_m$ and set $A=\{p_1,\ldots,p_M\}$. Denote by $\mathcal P(x)=p_m$ for $x\in P_m$. Then
\begin{align}\label{ineq-A-Px-Linfty}
  d\big(x,\mathcal P(x)\big)<\e.
\end{align}
Let $\mathcal P^{F}(x)=\big(\mathcal P(gx)\big)_{g\in F}$ for $F\in F(G)$. Recall that we also use $d_F$ to denote the metric on $\X^F$ for $F\in F(G)$ (see \eqref{metric} for the definition).
By \eqref{ineq-A-Px-Linfty}, we have $d_{F_n}\big((gx)_{g\in F_n}, \mathcal P^{F_n}(x)\big)<\e$ for any $x\in\X$.
For any two distinct points $x,y\in S_n$, we have
\begin{align*}
  d_{F_n}\big(\mathcal P^{F_n}(x),\mathcal P^{F_n}(y)\big)
  &\ge d_{F_n}(x,y)-d_{F_n}\big((gx)_{g\in F_n}, \mathcal P^{F_n}(x)\big)-d_{F_n}\big((gy)_{g\in F_n}, \mathcal P^{F_n}(y)\big)\\
  &> 6\e-2\e=4\e.
\end{align*}
Hence the set
$$\mathcal P^{F_n}(S_n)=\{\mathcal P^{F_n}(x)| x\in S_n\}\subset A^{F_n}$$
is a $4\e$-separated set of $\X^{F_n}$ with respect to the metric $d_{F_n}$.
Moreover, since $\nu_{F_n}$ is the uniform distribution over $S_n$, the push-forward measure $\mathcal P^{F_n}_* \nu_n$ is also the uniform distribution measure over
$\mathcal P^{F_n}(S_n)$. Note that $|\mathcal P^{F_n}(S_n)|=|S_n|$.

Let $0<\alpha<\frac{1}{4}$. let $X: \Omega\rightarrow \X$ be a random variable defined on some probability space $(\Omega,\p)$ such that the law of $X$ is given by $\mu$.
For $F\in F(G)$, let $Y_{F,g}: \Omega\rightarrow \X$ ($g\in F$) be random variables defined on the same probability space $(\Omega,\p)$ such that
$Y_F=(Y_{F,g})_{g\in F}$  and $X$ are $(F,\e,\alpha)_{L^{\infty}}$-close, i.e.
\begin{align}\label{ineq-4-XY-Linfty}
  \E\bigg(\frac{1}{|F|}\#\{g\in F:d(gX,Y_{F,g})\ge\e\}\bigg)<\alpha.
\end{align}
We can assume the distribution of $Y_F$ is supported on a finite set $\Y_F\subset \X^F$.
By the Data-processing inequality, $$I(X;Y_F)\ge I\big(\mathcal P^F(X);Y_F\big).$$
Let $\tau_F={\rm Law}\big(\mathcal P^F(X),Y_F\big)$ be the law of $(\mathcal P^F(X),Y_F)$, which is supported on $A^F\times \Y_F$.
Since $d\big(gX,\mathcal P(gX)\big)<\e$, it follows that
\begin{align*}
  \big\{g\in F:d\big(\mathcal P(gX),Y_{F,g}\big)\ge2\e\big\}\subset \{g\in F:d(gX,Y_{F,g})\ge\e\}.
\end{align*}
Denote by $f_F(x,y)=\#\{g\in F: d(x_g,y_g)\ge 2\e\}$ for $x=(x_g)_{g\in F}\in A^F$ and $y=(y_g)_{g\in F}\in \Y_F$.
Thus
\begin{align}\label{ineq-A-tau}
  \E_{\tau_F}f_F(x,y)&:=
  \int_{A^F\times \Y_F}f_F(x,y)d\tau_F(x,y)\nonumber\\
  &=\E\bigg(\#\big\{g\in F: d\big(\mathcal P(gX),Y_{F,g}\big)\ge2\e\big\}\bigg)\nonumber\\
  &<\alpha |F|.
\end{align}

For each $n\ge 1$, we consider the couplings of $(\mathcal P^F_*\mu_n,\mathcal P^F_*\mu)$. Choose a probability measure $\pi_{F,n}$ that minimizes the following integral
  $$\int_{A^F\times A^F}\bar d_F(x,y)d\pi(x,y)$$
  among all such couplings $\pi$. Also similar to Claim 30 of \cite{LT2}, the sequence $\pi_{F,n_i}$ converges to $(\mathcal P^F\times\mathcal P^F)_*\mu$ in the weak$^*$ topology.

Compose $\pi_{F,n}$ and $\tau_F$ to produce a coupling $\tau_{F,n}$ of $\big(\mathcal P^F_*\mu_n,{\rm Law}(Y_F)\big)$ by the following way:
$$\tau_{F,n}(x,y)=\sum_{x'\in A^F}\pi_{F,n}(x,x')\p\big(Y_F=y|\mathcal P^F(X)=x'\big), \qquad (x\in A^F,y\in\Y_F).$$
We note here that the sequence $\tau_{F,n_i}$ converges to $\tau_F$ in the weak$^*$ topology and hence by \eqref{ineq-A-tau},
\begin{align}\label{ineq-A-tau2}
  \E_{\tau_{F,n_i}}f_F(x,y)=\int_{A^F\times \Y_F}f_F(x,y)d\tau_{F,n_i}(x,y)<\alpha |F|
\end{align}
for all sufficiently large $n_i$.

Similar to the proof of Proposition \ref{prop-upper}, for $x\in \bigcup_{g\in F_n}\mathcal P^F(gS_n)$ and $y\in \X^F$, we define a conditional probability mass function $\tau_{F,n}(y|x)$ by
$$\tau_{F,n}(y|x)=\frac{\tau_{F,n}(x,y)}{\mathcal P^F_*\mu_n(x)}.$$

For any $K\in F(G)$ with $e_G\in K$ and $0<\e_1<\alpha$, as in Proposition \ref{prop-upper}, by Lemma \ref{lemma-folner}, there exists $\mathcal T$, a finite tiling of $G$,
satisfying conditions (C1) and (C2) in Proposition \ref{prop-upper}:
 \begin{enumerate}
  \item[(C1)] $\mathcal{T}$ has shapes $\{F_{m_1},\ldots,F_{m_l}\}$ consisted with F{\o}lner sets in $\{F_n\}$ and each $F_{m_j}$
  is $(K,\e_1)$-invariant;
  \item[(C2)] for sufficiently large $i$ (hence $F_{n_i}\in F(G)$ is sufficiently invariant), for each $1\le j\le l$, the family of sets $\{C_jg^{-1}\}_{g\in F_{n_i}}$ covers a subset $\tilde F_{n_i}\subset F_{n_i}$ with
  $|\tilde F_{n_i}|>(1-\e_1)|F_{n_i}|$ at most $(1+\e_1)\rho_{\mathcal T}(F_{m_j},F_{n_i})\frac{|F_{n_i}|}{|F_{m_j}|}$-many times, where $C_j$ is the center of the shape $F_{m_j}$.
\end{enumerate}

For $a\in \X$, $x=(x_g)_{g\in F_{n_i}}\in \mathcal P^{F_{n_i}}(S_{n_i})$ and $g\in F_{n_i}$, we define probability mass functions $\sigma_{F_{n_i},g}(\cdot|x)$ and $\sigma_{F_{n_i}}(\cdot|x)$ on $\X^{F_{n_i}}$ as exactly as \eqref{def-measure1} and \eqref{def-measure2} respectively:\\
for $y=(y_g)_{g\in F_{n_i}}\in \X^{F_{n_i}}$,
\begin{align}\label{def-measure1-Linfty}
  \sigma_{F_{n_i},g}(y|x)=\prod_{j=1}^l\prod_{c\in C_j,F_{m_j}c\subset F_{n_i}g,cg^{-1}\in \tilde{F}_{n_i}}\tau_{F_{m_j},n_i}(y_{F_{m_j}cg^{-1}}|x_{F_{m_j}cg^{-1}})\cdot \prod_{h\in R_g}\delta_a(y_h)
\end{align}
and 
\begin{align}\label{def-measure2-Linfty}
  \sigma_{F_{n_i}}(y|x)=\frac{1}{|F_{n_i}|}\sum_{g\in F_{n_i}}\sigma_{F_{n_i},g}(y|x).
\end{align}

Here we recall that $$y_{F_{m_j}cg^{-1}}=(y_h)_{h\in F_{m_j}cg^{-1}}\in \X^{F_{m_j}cg^{-1}},$$
$$x_{F_{m_j}cg^{-1}}=(x_h)_{h\in F_{m_j}cg^{-1}}\in \mathcal P^{F_{m_j}cg^{-1}}(S_{n_i})$$
and
$$R_g=F_{n_i}\setminus \bigg(\bigsqcup_{j=1}^l\bigsqcup_{c\in C_j,F_{m_j}c\subset F_{n_i}g,cg^{-1}\in \tilde{F}_{n_i}}F_{m_j}cg^{-1}\bigg).$$
Moreover, when $n_i$ is large enough, $|R_g|< \e_1 |F_{n_i}|$.

Then as exactly as Claim \ref{claim-I}, when $n_i$ is large enough, there exists some $1\le j\le l$ such that
\begin{align}\label{claim-I-Linfty}
  (1-\e_1)\frac{1}{|F_{n_i}|}I(\mathcal P^{F_{n_i}}_*\nu_{n_i},\sigma_{F_{n_i}})\le \frac{1}{|F_{m_j}|}I\big(\mathcal P^{F_{m_j}}_*(\mu_{n_i}),\tau_{F_{m_j},n_i}\big).
\end{align}

Denote by $\E_{\mathcal P^{F_{n_i}}_*\nu_{n_i},\sigma_{F_{n_i}}}f_{F_{n_i}}(x,y)$ the expected value of the function $f_{F_{n_i}}(x,y)$ with respect to the probability measure
$\mathcal P^{F_{n_i}}_*\nu_{n_i}(x)\sigma_{F_{n_i}}(y|x)$.

\begin{claim}\label{claim-E-Linfty}
For sufficiently large $n_i$,
\begin{align*}
  \E_{\mathcal P^{F_{n_i}}_*\nu_{n_i},\sigma_{F_{n_i}}}f_{F_{n_i}}(x,y)< 3\alpha|F_{n_i}|.
\end{align*}
\end{claim}
\begin{proof}[{\bf Proof of Claim \ref{claim-E-Linfty}}]
  By \eqref{def-measure1-Linfty} and \eqref{def-measure2-Linfty}, the definitions of probability mass functions $\sigma_{F_{n_i},g}(\cdot|x)$ ($g\in F_{n_i}$) and $\sigma_{F_{n_i}}(\cdot|x)$, we have
\begin{align*}
  \E_{\mathcal P^{F_{n_i}}_*\nu_{n_i},\sigma_{F_{n_i}}}f_{F_{n_i}}(x,y)=\frac{1}{|F_{n_i}|}\sum_{g\in F_{n_i}}\E_{\mathcal P^{F_{n_i}}_*\nu_{n_i},\sigma_{F_{n_i},g}} f_{F_{n_i}}(x,y)
\end{align*}
and
\begin{align*}
  &\E_{\mathcal P^{F_{n_i}}_*\nu_{n_i},\sigma_{F_{n_i},g}}f_{F_{n_i}}(x,y)\\
  \le & \sum_{j=1}^l\sum_{c\in C_j,F_{m_j}c\subset F_{n_i}g,cg^{-1}\in\tilde F_{n_i}}\E_{\mathcal P^{F_{m_j}}_*((cg^{-1})_*\nu_{n_i}),\tau_{F_{m_j},n_i}}
  f_{F_{m_j}}(x',y')+|R_g|,
\end{align*}
where $x,y$ are random points in $\X^{F_{n_i}}$ and $x',y'$ appear in $f_{F_{m_j}}(x',y')$ are in $\X^{F_{m_j}}$.

When $F_{n_i}$ is sufficiently invariant, $|R_g|<\e_1|F_{n_i}|$. Hence
\begin{align*}
  &\E_{\mathcal P^{F_{n_i}}_*\nu_{n_i},\sigma_{F_{n_i}}}f_{F_{n_i}}(x,y)\\
  \le& \frac{1}{|F_{n_i}|}\sum_{j=1}^l\sum_{g\in F_{n_i}}\sum_{c\in C_j,F_{m_j}c\subset F_{n_i}g,cg^{-1}\in\tilde F_{n_i}}
  \E_{\mathcal P^{F_{m_j}}_*((cg^{-1})_*\nu_{n_i}),\tau_{F_{m_j},n_i}}f_{F_{m_j}}(x',y')+\e_1|F_{n_i}|\\
  \le& \frac{1}{|F_{n_i}|}\sum_{j=1}^l\sum_{h\in \tilde F_{n_i}}(1+\e_1)t_j\frac{|F_{n_i}|}{|F_{m_j}|}\E_{\mathcal P^{F_{m_j}}_*(h_*\nu_{n_i}),\tau_{F_{m_j},n_i}}f_{F_{m_j}}(x',y')+\e_1|F_{n_i}|\\
  &\qquad \text{ (by condition (C2) and recall here } t_j=\rho_{\mathcal T}(F_{m_j},F_{n_i})\text{)}\\
  \le& \sum_{j=1}^l(1+\e_1)t_j\frac{|F_{n_i}|}{|F_{m_j}|}\E_{\mathcal P^{F_{m_j}}_*(\frac{1}{|F_{n_i}|}\sum_{h\in F_{n_i}}h_*\nu_{n_i}),\tau_{F_{m_j},n_i}}f_{F_{m_j}}(x',y')+\e_1|F_{n_i}|\\
  =&\sum_{j=1}^l(1+\e_1)t_j\frac{|F_{n_i}|}{|F_{m_j}|}\E_{\mathcal P^{F_{m_j}}_*\mu_{n_i},\tau_{F_{m_j},n_i}}f_{F_{m_j}}(x',y')+\e_1|F_{n_i}|\\
  =&\sum_{j=1}^l(1+\e_1)t_j\frac{|F_{n_i}|}{|F_{m_j}|}\int_{A^{F_{m_j}}\times \Y_{F_{m_j}}}f_{F_{m_j}}(x',y')d\tau_{F_{m_j},n_i}(x,y)+\e_1|F_{n_i}|\\
  =&\sum_{j=1}^l(1+\e_1)t_j\frac{|F_{n_i}|}{|F_{m_j}|}\E_{\tau_{F_{m_j},n_i}}f_{F_{m_j}}(x',y')+\e_1|F_{n_i}|.
\end{align*}

Recall that $0<\e_1<\alpha<\frac{1}{4}$ and $\sum_{j=1}^lt_j\le 1$. By \eqref{ineq-A-tau2}, for sufficiently large $n_i$, we have
\begin{align*}
  \E_{\tau_{F_{m_j},n_i}}f_{F_{m_j}}(x',y')<\alpha|F_{m_j}|, \text{ for each }1\le j\le l.
\end{align*}
Hence for sufficiently large $n_i$,
\begin{align*}
  \E_{\mathcal P^{F_{n_i}}_*\nu_{n_i},\sigma_{F_{n_i}}}f_{F_{n_i}}(x,y)&< \big((1+\e_1)\alpha+\e_1\big)|F_{n_i}|\\
  &<3\alpha|F_{n_i}|.
\end{align*}
This finishes the proof of Claim \ref{claim-E-Linfty}.
\end{proof}

Note that the set $\mathcal P^{F_{n_i}}(S_{n_i})=\{\mathcal P^{F_{n_i}}(x)| x\in S_{n_i}\}\subset A^{F_{n_i}}$ ($|\mathcal P^{F_{n_i}}(S_{n_i})|=|S_{n_i}|$)
is a $4\e$-separated set of $\X^{F_{n_i}}$ with respect to the metric $d_{F_{n_i}}$ and
the push-forward measure $\mathcal P^{F_{n_i}}_* \nu_{n_i}$ is the uniform distribution measure over
$\mathcal P^{F_{n_i}}(S_{n_i})$.
By Claim \ref{claim-E-Linfty} and Lemma \ref{lemma-mutual-Linfty}, for sufficiently large $n_i$,
\begin{align}
  \frac{1}{|F_{n_i}|}I(\mathcal P^{F_{n_i}}_*\nu_{n_i},\sigma_{F_{n_i}})\ge \frac{1}{|F_{n_i}|}\log |S_{n_i}|-3\alpha\log M-H(3\alpha).
\end{align}

It follows from \eqref{ineq-A-Sn-Linfty}, \eqref{claim-I-Linfty} and Claim \ref{claim-E-Linfty} that for sufficiently large $n_i$,
there exists $1\le j\le l$ ($j$ depends on $n_i$ and $l$ is independent on $n_i$) such that
\begin{align*}
  &\frac{1}{|F_{m_j}|}I\big(\mathcal P^{F_{m_j}}_*(\mu_{n_i}),\tau_{F_{m_j},n_i}\big)\\
  \ge &(1-\e_1)\big(\frac{1}{|F_{n_i}|}\log |\#(\X,d_{F_{n_i}},12\e)|-3\alpha\log M-H(3\alpha)\big ).
\end{align*}

By choosing some subsequence of $\{n_i\}$ (we still denote it by $\{n_i\}$), for some $1\le j\le l$, the probability measures $\tau_{F_{m_j},n_i}$
converge to  $\tau_{F_{m_j}}={\rm Law}\big(\mathcal P^{F_{m_j}}(X),Y_{F_{m_j}}\big)$ in the weak$^*$ topology. Let $n_i\rightarrow \infty$.
By (3) of Proposition \ref{prop-mutual1},
\begin{align*}
  \frac{1}{|F_{m_j}|}I\big(\mathcal P^{F_{m_j}}(X);Y_{F_{m_j}}\big)\ge (1-\e_1)\big(S(\X,G,d,12\e)-3\alpha\log M-H(3\alpha)\big ).
\end{align*}
By (1) of Proposition \ref{prop-mutual1}, the data-processing inequality,
\begin{align*}
  \frac{1}{|F_{m_j}|}I(X;Y_{F_{m_j}})\ge (1-\e_1)\big(S(\X,G,d,12\e)-3\alpha\log M-H(3\alpha)\big ).
\end{align*}

Let $(K,\e_1)$ be chosen from the pairs $(K_n,\frac{1}{n})$, where $K_n\in F(G)$ and $K_n\uparrow G$. The $F_{m_j}$'s above subject to $(K_n,\frac{1}{n})$ form a new F{\o}lner sequence. We denote this new F{\o}lner sequence by $\{T_n\}$ and then it follows that
\begin{align*}
  \underline R_{\mu,\infty}(\{T_n\},\e,\alpha)\ge S(\X,G,d,12\e)-3\alpha\log M-H(3\alpha).
\end{align*}
Let $\alpha\rightarrow 0$. Noting that $R_{\mu,\infty}(\e)$ is independent on the choice of F{\o}lner sequences, we have
\begin{align*}
  R_{\mu,\infty}(\e)\ge S(\X,G,d,12\e).
\end{align*}
This completes the proof of Proposition \ref{prop-upper-Linfty}.
\end{proof}

Theorem \ref{theorem-Linfty} then follows from Lemma \ref{lemma-lower-Linfty} and Proposition \ref{prop-upper-Linfty}.


\begin{thebibliography}{99}
\bibitem{ADP} Y. Ahn, D. Dou and K. K. Park, Entropy dimension and its variational principle, {\it Studia Math.}, {\bf 199} (2010), no. 3, 295-309.
\bibitem{Ca} M. De Carvalho, Entropy dimension of dynamical systems, {\it Portugal. Math.}, {\bf 54} (1997), no. 1, 19-40.
\bibitem{C} M. Coornaert, Topological dimension and dynamical systems, Universitext, Springer, 2015. Translation from the French language edition: Dimension topologique et syst\`{e}mes dynamiques by M. Coornaert, Cours sp\'{e}cialis\'{e}s 14, Soci\'{e}t\'{e} Math\'{e}matique de France, Paris, 2005.
\bibitem{CK} M. Coornaert and F. Krieger, Mean topological dimension for actions of discrete amenable groups, Discrete Contin. Dyn. Syst. 13 (2005) 779--793.
\bibitem{D} D. Dou, Minimal subshifts of arbitrary mean topological dimension, Discrete Contin. Dyn. Syst. 37 (2017), no. 3, 1411-1424.
\bibitem{DHP1} D. Dou, W. Huang and K. K. Park, Entropy dimension of topological dynamics, {\it Trans. Amer. Math. Soc.}, {\bf 363} (2011), 659-680.
\bibitem{DHP2} D. Dou, W. Huang and K. K. Park, Entropy dimension of measure preserving systems, {\it Trans. Amer. Math. Soc.}, {\bf 371} (2019), 7029-7065.
\bibitem{DHZ} Tomasz Downarowicz, Dawid Huczek and Guohua Zhang, Tilings of amenable groups, J. Reine Angew. Math., 747 (2019), 277-298.
\bibitem{FP}S. Ferenczi and K. K. Park, Entropy dimensions and a class of constructive examples, {\it Discrete Cont. Dyn. Syst.}, {\bf 17} (2007), no. 1, 133-141.
\bibitem{Go} T N T Goodman, Topological sequence entropy, Proceedings of the London Mathematical Society, 1974, 3(2): 331-350.
\bibitem{Gr} M. Gromov, Topological invariants of dynamical systems and spaces of holomorphic maps, Part I, Math. Phys. Anal. Geom. 2 (1999), 323--415.
\bibitem{Gu1} Y. Gutman, Embedding topological dynamical systems with periodic points in cubical shifts, Ergodic Theory and Dynamical Systems, 2015: 1-27.
\bibitem{Gu2} Y. Gutman, Mean dimension and Jaworski-type theorems, Proceedings of the London Mathematical Society, 2015, 111(4): 831-850.
\bibitem{GLT} Y. Gutman, E. Lindenstrauss and M. Tsukamoto, Mean dimension of $\Z^k$-actions, Geometric and Functional Analysis, 2016, 26(3): 778-817.
\bibitem{GT} Y. Gutman and M. Tsukamoto, Mean dimension and a sharp embedding theorem: extensions of aperiodic subshifts, Ergodic Theory and Dynamical Systems, 2014, 34(06): 1888-1896.
\bibitem{H} B. Hayes, Metric mean dimension for algebraic actions of sofic groups, Transactions of the American Mathematical Society, 2017.
\bibitem{KD} T. Kawabata and A. Dembo, The rate-distortion dimension of sets and measures, IEEE transactions on information theory, 1994, 40(5): 1564-1572.
\bibitem{KL} D. Kerr and H. Li, Ergodic theory: independence and dichotomies, Springer, 2017.
\bibitem{K1} F. Krieger, Groupes moyennables, dimension topologique moyenne et sous-d\'{e}calages, Geom.
Dedicata 122 (2006), 15--31.
\bibitem{K2} F. Krieger, Minimal systems of arbitrary large mean topological dimension, Israel J. Math. 172 (2009) 425--444.
\bibitem{Ku} A. G. Kushnirenko, On metric invariants of entropy type, {\it Russian Mathematical Surveys}, {\bf 22} (1967), no. 5, 53-61.
\bibitem{Li} H. Li, Sofic mean dimension, Advances in Mathematics, 2013, 244: 570-604.
\bibitem{LiL} Li H, Liang B. Mean dimension, mean rank, and von Neumann¨CL¨¹ck rank, Journal f¨¹r die reine und angewandte Mathematik (Crelles Journal), 2013.
\bibitem{LiL2} H. Li and B. Liang, Sofic mean length, arXiv:1510.07655, 2015.
\bibitem{L} E. Lindenstrauss, Mean dimension, small entropy factors and an embedding theorem, Inst. Hautes \'{E}tudes Sci. Publ. Math. 89 (1999), 227--262.
\bibitem{L2} E. Lindenstrauss, Pointwise theorems for amenable groups, Invent. Math. 146 (2001) 259--295.
\bibitem{LT} E. Lindenstrauss and M. Tsukamoto, Mean dimension and an embedding problem: an example, Israel J. Math. 199 (2014), no. 2, 573--584.
\bibitem{LT2} Lindenstrauss E, Tsukamoto M. From rate distortion theory to metric mean dimension: variational principle, IEEE Trans. Inform. Theory 64 (2018), no. 5, 3590--3609.
\bibitem{LW} E. Lindenstrauss and B. Weiss, Mean topological dimension, Israel J. Math. 115 (2000) 1--24.
\bibitem{Mis} M. Misiurewicz, A short proof of the variational principle for a $\Z^n_+$ action on a compact space, Ast\'{e}risque 40 (1976), 147--187.
\bibitem{OW} D.S. Ornstein, B. Weiss, Entropy and isomorphism theorems for actions of amenable groups, J. Anal. Math. 48(1987) 1--141.
\bibitem{T1} M. Tsukamoto, A packing problem for holomorphic curves, Nagoya Mathematical Journal, 2009, 194: 33-68.
\bibitem{T2} M. Tsukamoto, Gauge theory on infinite connected sum and mean dimension, Mathematical Physics, Analysis and Geometry, 2009, 12(4): 325-380.
\bibitem{T3} M. Tsukamoto, Large dynamics of Yang--Mills theory: mean dimension formula, arXiv:1407.2058, 2014.
\bibitem{V1} A. M. Vershik, Four definitions of the scale of an automorphism, Functional Analysis and Its Applications, 1973, 7(3): 169-181.
\bibitem{V2} A. M. Vershik, Dynamic theory of growth in groups: entropy, boundaries, examples, {\it Uspekhi Mat. Nauk}, {\bf 55} (2000), no. 4(334), 59-128,
translation in {\it Russian Math. Surveys}, {\bf 55} (2000), no. 4, 667-733.
\bibitem{V3} A. M. Vershik and A. D. Gorbulsky, Scaled entropy of filtrations of ¦Ò-fields, Theory of Probability and Its Applications, 2008, 52(3): 493-508.
\bibitem{WZ} T. Ward and Q. Zhang, The Abramov-Rokhlin entropy addition formula for amenable group actions, Monatsh. Math. 114 (1992), 317-329.
\bibitem{W} B. Weiss, Actions of amenable groups, Topics in Dynamics and Ergodic Theory. (2003) 226--262. London Math. Soc. Lecture Note Ser., 310, Cambridge Univ. Press, Cambridge, 2003.
\bibitem{ZCY} D. Zheng, E. Chen and J. Yang, On large deviations for amenable group actions, Discrete Contin. Dyn. Syst., 36 (2016), no. 12, 7191-7206.

\end{thebibliography}
\end{document}